\documentclass[hidelinks,onefignum,onetabnum]{siamart220329}

\usepackage{graphicx}
\usepackage{psfrag} 
\usepackage{mathptmx}      
\usepackage{amssymb}
\usepackage{amsmath}
\usepackage{algorithm}
\usepackage{marvosym}
\usepackage{a4wide}
\usepackage[noend]{algpseudocode}

\usepackage{color}

\usepackage[mathscr]{euscript}

\title{Shared ancestry graphs and symbolic arboreal maps}

\author{Katharina T. Huber\thanks{School of Computing Sciences, University of East Anglia, Norwich, UK.}
\and Vincent Moulton\footnotemark[1]
\and Guillaume E. Scholz\thanks{Bioinformatics Group, Department of Computer Science and Interdisciplinary Center for Bioinformatics, Universit\"at
				Leipzig, Leipzig, Germany. \email{guillaume@bioinf.uni-leipzig.de}}}

\begin{document}

\maketitle
			


\begin{abstract}
A {\em network} $N$ on a finite set $X$,  $|X|\geq 2$, is a connected directed acyclic graph with leaf set $X$ in which every root in $N$ has outdegree at least 2 and no vertex in $N$ has indegree and outdegree equal to 1; $N$ is {\em arboreal}  if the underlying unrooted, undirected graph of $N$ is a tree. Networks are of interest in evolutionary biology since they are used, for example, to represent the evolutionary history of a set $X$ of species whose ancestors have exchanged genes in the past. For $M$ some arbitrary set of symbols, $d:{X \choose 2} \to M \cup \{\odot\}$  is a {\em symbolic arboreal map} if there exists some arboreal network $N$ whose vertices with outdegree two or more are labelled by elements in $M$ and so that $d(\{x,y\})$, $\{x,y\} \in {X \choose 2}$, is equal to the label of the least common ancestor of $x$ and $y$ in $N$ if this exists and $\odot$ else. Important examples of symbolic arboreal maps include the symbolic ultrametrics, which arise in areas such as game theory, phylogenetics and cograph theory. In this paper we show that a map $d:{X \choose 2} \to M \cup \{\odot\}$ is a symbolic arboreal map if and only if $d$ satisfies certain 3- and 4-point conditions and the graph with vertex set $X$ and edge set consisting of those pairs $\{x,y\} \in {X \choose 2}$ with $d(\{x,y\}) \neq \odot$ is {\em Ptolemaic} (i.e. its shortest path distance satisfies Ptolemy's inequality). To do this, we introduce and prove a key theorem concerning the  {\em shared ancestry graph} for a network $N$ on $X$, where this is the graph with vertex set $X$ and edge set consisting of those $\{x,y\} \in {X \choose 2}$ such that $x$ and $y$ share a common ancestor in $N$. In particular, we show that for any connected graph $G$ with vertex set $X$ and edge clique cover $K$ in which there are no two distinct sets in $K$ with one a subset of the other, there is some network with $|K|$ roots and leaf set $X$ whose shared ancestry graph is $G$. 
\end{abstract}

\section{Introduction}

Given a finite set $X$, $|X|\geq 2$, an arbitrary non-empty set $M$ of {\em symbols}, and some element $\odot$ that is not in $M$, a
{\em symbolic map} is a function $d$ that maps the collection of 2-subsets of $X$, i.e. ${X \choose 2}$,
into the set $M^{\odot}= M \cup \{\odot\}$. For brevity,  
given a symbolic map $d$ we denote $d(\{x,y\})$, $\{x,y\} \in {X \choose 2}$, by $d(x,y)$. 
Important examples of such maps are the {\em symbolic ultrametrics}. These are maps $d:{X\choose 2}\to M$ 
for which there exists some rooted  $T$ with leaf set $X$ 
in which each internal vertex of $T$ is labelled by an element in $M$, 
and such that $d(x,y)$, $\{x,y\} \in {X \choose 2}$, 
is given by the element in $M$ that labels the least common ancestor of $x$ and $y$ in $T$
(see e.g. Figure~\ref{fig-new}(i)). 
Symbolic ultrametrics were introduced in a different guise by Gurvich in \cite{G84}, and subsequently rediscovered and
studied in \cite{BD98}. They are a generalization
of the well-known {\em ultrametrics} (see e.\,g.\,\cite{SS03}), 
and have close links with the theory of cographs (see e.\,g.\,\cite{HHHMSW13}).

\begin{figure}[h]
	\begin{center}
		\includegraphics[scale=0.8]{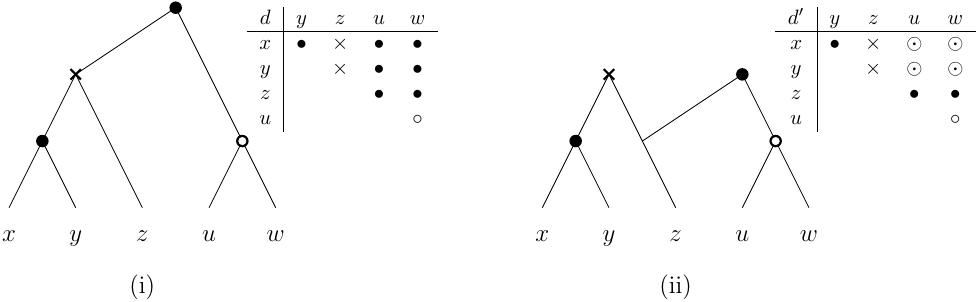}
	\end{center}
	\caption{For the set $M=\{\bullet,\circ,\times\}$, 
		(i) a phylogenetic tree with leaf set $X=\{x,y,z,u,w\}$, a labelling of its internal vertices by $M$,
		and the corresponding symbolic ultrametric $d$.
		(ii) An arboreal network with leaf set $X$, a labelling of its internal vertices having outdegree 2 by $M$,
		and the corresponding symbolic arboreal map $d'$.}
	\label{fig-new}
\end{figure}

Symbolic maps also  arise from more general structures than trees.
For example, maps arising from hypergraphs and di-cographs
are investigated in \cite{G00} and \cite{hellmuth2017mathematics}, respectively (see also e.g. \cite{geiss2020best}). 
In this paper, we are interested in understanding symbolic maps
that arise from a {\em network on $X$}, that is, 
a connected directed acyclic graph with leaf set $X$ 
in which every root in $N$ has outdegree at least 2 and no vertex in $N$ has indegree and outdegree equal to 1.
Networks arise, for example, in the study of the evolutionary history of 
species whose ancestors have exchanged genes in the past (see e.g. \cite{huber2022forest}),
and important examples include the well-studied {\em phylogenetic networks}, that is, networks that have a single root
(see e.g. \cite{S16} for a recent review).  Relatively little is known concerning properties 
of symbolic maps arising from networks; to our best knowledge they have only been 
directly considered in \cite{bruckmann2022modular} where 
symbolic maps arising from rooted median networks are introduced, and 
in \cite{HS18} where some results are presented for 3-way symbolic maps 
that arise from so-called level-1 networks.  

Here we shall consider symbolic maps that arise from
{\em arboreal networks}, that is, networks whose
underlying (undirected and unrooted) graph is a tree \cite{huber2022forest}.
An example of an arboreal network is pictured in Figure~\ref{fig-new}(ii); note
that such a network has a single root if and only if it is a rooted tree.
As with symbolic ultrametrics, symbolic maps arise naturally from  
arboreal networks by labelling each vertex in such a network
with outdegree at least 2 by an element in $M$.
In particular, a symbolic map $d$ is obtained from an arboreal network $N$ by defining $d(x,y)$, $\{x,y\}\in {X \choose 2}$, to be 
the element in $M$ that labels the least common ancestor of $x$ and $y$ in $N$ if
such a vertex exists, and $\odot$ otherwise (see e.g. Figure~\ref{fig-new}(ii)).

In this paper, we characterize {\em symbolic arboreal maps}, that is, 
symbolic maps that arise from arboreal networks.
Note that symbolic ultrametrics can be characterized amongst symbolic maps $d$ in terms of a 3- and 4-point condition
as follows \cite{BD98,G84}. The 3-point condition states that there are no 
$x,y,z \in X$ distinct such that $|\{d(x,y), d(x,z),d(y,z)\}| = 3$
and $\odot \not\in \{d(x,y), d(x,z),d(y,z)\}$, and the 4-point condition states that there are 
no four distinct elements $x,y,z,u$ in $X$ such that
$$
d(x,y)=d(y,z)=d(z,u) \neq d(y,u)=d(u,x)=d(x,z),
$$
and $\odot \not\in \{d(x,y), d(x,z)\}$\footnote{We have stated the 3- and 4-point conditions in slightly 
more general terms than in \cite{BD98,G84} as we need to consider 
the additional $\odot$ symbol which does not arise when considering only trees.}.
In our main result, Theorem~\ref{thm-diss}, we show that a symbolic map
is arboreal if and only if it satisfies these 3- and 4-point conditions, an
additional 4-point condition, and the graph $G_d$ with vertex set $X$ and
edges consisting of elements $\{x,y\}\in {X\choose 2}$, with $d(x,y)\not=\odot$ is {\em Ptolemaic}.
Note that a graph with vertex set $X$ is Ptolemaic if its shortest path distance $d^*$ satisfies
Ptolemy's inequality \cite{howorka1981characterization}, i.e. 
	$$
	d^*(x,y) \cdot d^*(z,u) + d^*(x,u) \cdot d^*(y,z) \ge d^*(x,z) \cdot d^*(y,u)
	$$
holds for all $x,y,z,u \in X$. In addition, we show that there is a 
special type of labelled arboreal network that can be used to uniquely represent 
any given symbolic arboreal map (see Theorem~\ref{cor-discr}).
	
The rest of this paper is organised as follows. In Section~\ref{sec-prem}, we collect 
together relevant basic definitions and terminology. In Section~\ref{sec-arb}, we then formally 
define arboreal networks and present some characterizations of such networks that 
will be useful later on. In Section~\ref{sec-sag}, we introduce the notion 
of the {\em shared ancestry graph} for a network, 
and show that  given any connected graph $G$ with vertex set $X$, we can construct a network $N$ with leaf set $X$
from any edge clique cover of $G$ that {\em represents} $G$, that is, whose shared ancestry graph is $G$ (Theorem~\ref{thm-unirep}).
In Section~\ref{sec-ptol}, we review some properties of Ptolemaic graphs, including a key
result concerning the laminar structure of Ptolemaic graphs from \cite{UU09},  and
show that the minimum size of a edge clique cover for such a connected graph
is equal to the number of maximal cliques in that graph with size at least 2 (Theorem~\ref{lm-prime}).
We then use these results in Section~\ref{sc-abrep} to characterize shared ancestry graphs of 
arboreal networks, showing that 
if $G$ is a connected graph with vertex set $X$ then there exists an arboreal network with 
leaf set $X$ that represents $G$ if and only if $G$ is Ptolemaic (Theorem~\ref{thm-wide}).
In Section~\ref{sc-sym}, we prove our 
aforementioned main result (Theorem~\ref{thm-diss}) by linking properties of the shared 
ancestry graph of an arboreal network whose associated
symbolic map is $d$ with the graph $G_d$ as defined above. We also state 
the uniqueness result, Theorem~\ref{cor-discr}, which we prove in the Appendix.
We conclude in Section~\ref{sec-disc} by presenting some potential directions for future work.

\section{Preliminaries}\label{sec-prem}

Throughout this paper, $X$ is a finite set with $|X|\ge 2$, and all graphs are simple, 
directed or undirected graphs. To simplify terminology, we 
usually refer to a directed graph as a {\em digraph} and to an undirected graph as a graph. 
	
	Let $N$ be a digraph with vertex set $V(N)$. Then we call
	the number of arcs coming into a vertex $v$ of $N$ the 
	{\em indegree} of $v$ and denote it by $indeg_N(v) = indeg(v)$. Similarly, 
	we call the number of outgoing arcs of a vertex $v$ the {\em outdegree}
    of $v$ and denote it by $outdeg_N(v) =outdeg(v)$. A
{\em leaf}  of $N$ is a vertex with indegree 1 and outdegree 0, and  
a \emph{root} is a vertex with indegree 0. We denote the set of leaves of $G$ by $L(G)$.
An \emph{internal vertex} (of $N$) is a vertex with outdegree 1 or more, and a \emph{tree-vertex} (of $N$) 
is a vertex with indegree 0 or 1. 
Note that if $N$ contains a vertex $v$ with indegree and outdegree $1$, by \emph{suppressing} 
$v$ we mean that we remove $v$ and its incident arcs and 
add a new arc from the parent of $v$ to the child of $v$.
A vertex $v$ of $N$ is said to be an \emph{ancestor} of a vertex $w$ in $N$ 
if there exists a directed path in $N$ from $v$ to $w$. In this case, we say 
	that $w$ is {\em below} $v$ and call $w$ a \emph{descendant} of $v$. If $v$ is an ancestor of $w$ and 
	$v\not=w$ then we call $v$ a {\em strict} ancestor of $w$ and $w$ a {\em strict} descendant of $v$. Note that a vertex is 
both an ancestor and a descendant of itself. 
If neither $v$ nor $w$ is an ancestor of the other, then we say that $v$ and $w$ are \emph{incomparable} (in $N$). 
Note that if two vertices of $N$ are incomparable then they must necessarily be distinct.
We say that two vertices $v,w\in V(N)$ \emph{share an ancestor} in $N$ if there 
exists a vertex $u$ (possibly equal to $v$ or $w$) such that $u$ is an ancestor of both $v$ and $w$.  
We say that $N$ is {\em connected} if the underlying graph of $N$ obtained by ignoring the 
directions of the arcs of $N$ is a connected graph.
 
A \emph{network (on $X$)} is a connected, acyclic digraph $N$ with leaf set $X$
such that all vertices of $N$ of indegree 0 have 
outdegree at least 2, all vertices of outdegree 0 have indegree 1, and no 
vertices have indegree and outdegree equal to 1. 
For $N$ a network, we denote by $R(N)$ the set of roots of $N$, and
let $r(N)=|R(N)|$. For simplicity, we shall sometimes call a network with $k\geq 1$ 
roots a {\em $k$-rooted network}. For $v$ a  vertex of $N$,
we let $C(v) \subseteq X$ 
denote the set of leaves of $N$ that have $v$ as an ancestor.
A {\em (single rooted) phylogenetic network (on $X$)} is a 
network on $X$ with one root (see e.g. \cite{S16}), and a {\em phylogenetic tree (on $X$)} is a 
phylogenetic network in which every vertex is a tree-vertex.

Vertices in a network $N$ that have indegree $2$ or more are called \emph{hybrid vertices}, and the set of 
hybrid vertices of $N$ is denoted by $H(N)$. 
We put $h(N)=|H(N)|$. Also, we put $\tilde{h}(N)=0$ if  $H(N)=\emptyset$ and, 
otherwise, we put
$\tilde h(N)=\sum_{h \in H(N)} ( indeg_N(h)-1 )$.
Note that $\tilde h(N)=h(N)$ if and only if all hybrid vertices of $N$ have indegree 2. 
If $r(N) \geq 2$, then for $r \in R(N)$, we denote by $N-r$ the digraph obtained from $N$ 
by first removing all vertices of $N$ and their incident arcs that are not a descendant of any vertex in $R(N)-\{r\}$ and 
then suppressing resulting vertices of indegree and outdegree $1$. 
Note that, in general, $N-r$ need 
not be a network as it might not be connected.

\begin{figure}[h]
		\begin{center}
			\includegraphics[scale=0.8]{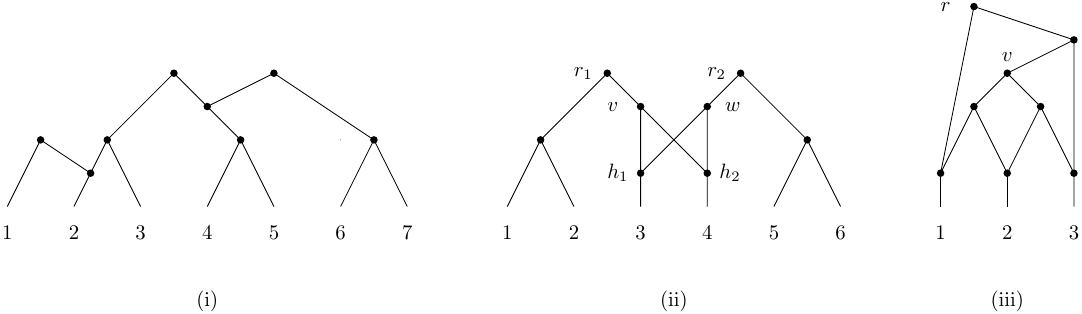}
		\end{center}
		\caption{(i) An arboreal network with $3$ roots on $\{1, \ldots, 7\}$. (ii) A 2-rooted 
			network on $\{1, \ldots, 6\}$ that is not arboreal as it contains 
			the $2$-alternating cycle $v,h_1,w,h_2$. (iii) A 1-rooted network (i.\,e.\,a
			phylogenetic network) on $\{1,2,3\}$ that contains 1-, 2- and 3-alternating cycles.}
		\label{fig-intro}
\end{figure}

\section{Characterizing arboreal networks}\label{sec-arb}

We call a network $N$  \emph{arboreal} if its underlying graph is a tree.
For example, the $3$-rooted network depicted in Figure~\ref{fig-intro}(i) is arboreal.
In this section, we give two characterizations of arboreal networks that will be
useful later on. We begin with a useful lemma.
	
\begin{lemma}\label{lm-rh1}
	Let $N$ be a network. Then $\tilde h(N) \geq r(N)-1$.
		In particular, $H(N)=\emptyset$ if and only if $N$ is a phylogenetic tree.
\end{lemma}
\begin{proof}
We show  that $\tilde h(N) \geq r(N)-1$ holds for all networks $N$ using induction on $r(N)$. 
Let $N$ be a network. Since $H(N)=\emptyset$ if and only if $N$ is a phylogenetic tree, the base case is
$r(N)=1$. If $r(N)=1$, then 
the inequality holds trivially since $\tilde{h}(N)=h(N)=0$.

Now, suppose that $r(N) \ge 2$.
We first claim that there must exist some $r \in R(N)$ such that $N-r$ is a network.
It suffices to prove that $N-r$ is connected.
Pick $r_1 \in R(N)$. If $N-r_1$ is connected, then  the claim follows as we can take $r$ to be $r_1$. Otherwise, we can pick some $r_2 \in R(N)-\{r_1\}$ such that $r_2$ is a vertex of a connected component $C_1$ of $N-r_1$ 
with the fewest number of roots amongst all 
connected components of $N-r_1$. If $N-r_2$ is connected, then the claim follows again as we can take $r$ to be $r_2$. Otherwise the correspondingly defined connected component $C_2$ has strictly 
fewer roots than $C_1$ and we can continue this process of picking a root with $r_1 $ replaced by $r_2$ and $r_2$ replaced by a root in $R(N)-\{r_1,r_2\}$. Since $R(N)$ is finite, this 
process of picking elements in $R(N)$ must eventually come to an end. This completes the proof of the claim.

Now, suppose that the inequality $\tilde h(N') \geq r(N')-1$
holds for all networks $N'$ with $r(N')<r(N)$. Consider 
a root $r$ of $N$ such that $N'=N-r$ is a 
network, which exists by the claim. Then $r(N')=r(N)-1$ and, because $N$ is connected, $\tilde h(N') < \tilde h(N)$. By our 
induction hypothesis, we have $\tilde h(N') \geq r(N')-1$, so $\tilde h(N) \geq r(N)-1$ follows.
\end{proof}

We now present two characterizations for arboreal networks, which  
we shall use later on without always explicitly referring them.
Let $N$ be a network. An \emph{$k$-alternating cycle} of $N$ is a 
sequence $v_1, h_1, v_2, \ldots v_k, h_k$, $k \geq 1$ of vertices of $N$ such that for all $1 \leq i \leq k$, $h_i$ 
is a hybrid vertex of $N$, and there exists internal vertex-disjoint directed paths from $v_i$ to $h_i$ and from $v_{i+1}$ to $h_i$, 
respectively (where we put $v_{k+1}=v_1$). 
For example, the sequence of vertices $v,h_1,w,h_2$ of the network depicted 
in Figure~\ref{fig-intro}(ii) is a 2-alternating cycle.  Note that $k$-alternating cycles are closely related to so 
called zig-zag paths introduced in \cite{Z16} and also to crowns as introduced in \cite{H21}.

\begin{proposition}
	\label{pr-arbalt}
Let $N$ be a network. Then the following statements are equivalent.
\begin{itemize}
	\item[(i)] $N$ is arboreal. 
	\item[(ii)] $\tilde h(N)=r(N)-1$.
	\item[(iii)] $N$ does not contain a $k$-alternating cycle for any $k \ge 1$.
\end{itemize}
\end{proposition}

\begin{proof}
(i)  $\Rightarrow$ (ii) Suppose that $N$ is an arboreal network. We show 
that $\tilde h(N)=r(N)-1$ using induction on $r(N)$. For the base case, if $r(N)=1$, then $N$ is a 
phylogenetic tree. So, $\tilde h(N)=0=r(N)-1$.

Now, suppose that $r(N)\ge 2$ and that the stated equality holds for all arboreal networks $N'$ with $r(N')<r(N)$. 
Consider a root $r$ of $N$ such that $N'=N-r$ is a network, which 
exists by the claim in the second paragraph of the proof of Lemma~\ref{lm-rh1}. 
Furthermore, $N'$ is arboreal and $r(N')=r(N)-1$. Also, 
$\tilde h(N')=\tilde h(N)-1$ since $N'$  has one root less than $N$ and so $N$ must have a unique hybrid vertex whose indegree decreased by precisely 1 in the construction of $N'$.
By induction hypothesis, it follows that $\tilde h(N)=r(N)-1$, as required.

(ii)  $\Rightarrow$ (i) Suppose for contradiction that $N$ is such that
 $\tilde{h}(N)=r(N)-1$ but $N$ is not arboreal. Then
there must exist a hybrid vertex $h$ in $H(N)$ and a parent $v \in V(N)$ of $h$ such that removing the incoming arc $(v,h)$ of $h$
does not disconnect $N$. Consider now the graph $N'$ obtained from $N$ by 
removing the arc $(v,h)$, introducing a new leaf $x$, adding the arc $(v,x)$, and 
suppressing $h$ if this has rendered it a vertex with indegree and outdegree 1. Since $N'$ is connected with leaf set $X \cup \{x\}$, $N'$ 
is a network on $X \cup \{x\}$. Furthermore, $r(N')=r(N)$ and $\tilde h(N')=\tilde h(N)-1$. By 
Lemma~\ref{lm-rh1}, $\tilde h(N') \geq r(N')-1$. Since $r(N')=r(N)$  it 
follows that $\tilde h(N)-1=\tilde{h}(N')\geq r(N')-1=r(N)-1=\tilde{h}(N)$; a contradiction.

(i)  $\Leftrightarrow$ (iii) It is straight-forward to check that the cycles in the 
underlying graph of $N$ are
in 1-1 correspondence with the $k$-alternating cycles of $N$, from which the 
equivalence of (i) and (iii) immediately follows.
\end{proof}

\section{The shared ancestry graph}\label{sec-sag}

Let $N$ be a network on $X$.
The \emph{shared ancestry graph $\mathcal A(N)$ (of $N$)} is the  graph 
whose vertex set is $X$ and in which two distinct vertices $x,y$ of $X$ are joined by an edge if and only 
if they share an ancestor in $N$. Note that since $N$ is connected, $\mathcal A(N)$ is also connected. 
In addition, note that if $N$ is a phylogenetic network then $\mathcal A(N)$ 
is a complete graph. However, the converse does not necessarily hold.
In this section, we shall prove that given any connected graph $G$ 
with vertex set $X$, we can construct a network $N$ with leaf set $X$
from any edge clique cover of $G$ whose shared ancestry graph is $G$.

We begin with some observations on shared ancestry graphs, and their
relationship with edge clique covers.
We say that a connected graph $G$ with vertex set $X$ is \emph{representable} 
if  there exists a network $N$ on $X$ such that $G$ is isomorphic to $\mathcal A(N)$ 
and that isomorphism is the identity on $X$. In that case, we also say that $N$ {\em represents} $G$.

\begin{proposition}\label{pr-univ}
Any connected graph $(X,E)$ is representable
by an $|E|$-rooted network on $X$.
\end{proposition}

\begin{proof}
	Suppose that $G=(X,E)$ is a connected graph.
We prove the proposition by constructing a $|E|$-rooted network $N$ on $X$ that represents $G$.

We initialize the construction of $N$ with the set of 
arcs $(x_p,x)$ where, for all $x \in X$,  we have that $x_p\not\in X$ and 
$x_p\not=y_p$, for all $x,y\in X$ distinct.
Then for all edges $e=\{x,y\}$ of $G$ taken in turn, we add to 
$N$ a vertex $v_e$, and two arcs $(v_e,x_p)$ and $(v_e,y_p)$. 
Since $G$ is connected, the digraph $N$ obtained once all 
edges of $G$ have been processed (and after all vertices of indegree and outdegree 1 
have been removed) is connected. Moreover, $N$ has leaf set $X$ and contains $|E|$ 
roots. Hence, $N$ is an $|E|$-rooted network on $X$. By construction, for 
any two distinct elements $x,y \in X$, there exists a vertex $v$ in $N$ 
that is an ancestor of $x$ and $y$ if and only if $\{x,y\}$ is an edge of $G$. Hence, $N$ represents $G$.
\end{proof}

Note that although the network $N$ constructed from $G$ in 
the proof of Proposition~\ref{pr-univ} is a network 
representing $G$ it is not necessarily the only
network on $X$ satisfying this property. Moreover, $N$ has many more roots than 
is usually necessary (viz. the number of edges in $G$).  
In the following, we present a way to construct a network representing any  connected graph $G$
with a minimum number of roots amongst all possible networks that represent $G$.

We begin with introducing some further terminology. For $G=(X,E)$ a  
graph and $\emptyset\not=Y \subseteq X$,
the \emph{subgraph $G[Y]$ of $G$ induced by $Y$} is the graph 
whose vertex set is $Y$ and any two vertices $u$ and $v$ in $Y$ are joined by an edge if $\{u,v\}\in E$.
%
For $G'$ a graph, we say that $G$ \emph{contains $G'$ (as an induced subgraph)} 
if there exists $Y \subseteq X$ such that $G'$ is isomorphic to $G[Y]$ and that isomorphism is the identity on $Y$.
A subset $Y\subseteq X$ is called a {\em clique (of $G$)} if $|Y|\geq 2$ and $\{x,y\}\in E$, 
for all $x,y\in Y$ distinct. 
If, in addition, there is no proper superset $Y'$ of $Y$ that is also a clique of $G$, then we say that $Y$ is a \emph{maximal clique} of $G$. 
Denoting by $\mathcal P(X)$ the powerset of $X$, we define $K(G) \subseteq \mathcal P(X)$ 
to be the set of all subsets of $X$ that are a maximal 
clique in $G$. Note that if $G$ does not contain isolated vertices, then 
each element of $X$ 
is contained in at least one set in $K(G)$.

Interestingly, if a network $N$ does not contain 3-alternating cycles then, as Lemma~\ref{lm-clique} shows,
the cliques in $\mathcal A(N)$ provide key information concerning the structure of $N$.

\begin{lemma}\label{lm-clique}
Let $N$ be a network on $X$  that does not contain 3-alternating cycles.
Let $Y \subseteq X$ with $|Y| \geq 2$. Then $Y$ is a clique in $\mathcal A(N)$ if and only if 
there exists a vertex 
in $N$ that is an ancestor of all leaves in $Y$.
\end{lemma}

\begin{proof}
One direction is trivial. Indeed, if all leaves in $Y$ share an 
ancestor in $N$ then any two elements in $Y$ are joined by an edge in $\mathcal A(N)$
by definition. Hence, $Y$ is a clique in $\mathcal A(N)$\footnote{Note that this direction 
	holds for all networks $N$, including networks containing $3$-alternating cycles.}.

Conversely, assume for contradiction that $N$ is a network on $X$ and that $Y \subseteq X$ 
	with $|Y|\geq 2$ is such that $Y$ is a clique in $\mathcal A(N)$ but no common ancestor in $N$ of the elements in $Y$ exists. 
Without loss of generality we may assume that $Y$ is such that for all subsets $Y'\subseteq Y$ 
with $|Y'|\geq 2$ there exists an ancestor in $N$ of all elements in $Y'$. 
Then $|Y|\geq 3$  as otherwise $Y$ is a clique of $\mathcal A(N)$ 
in the form of an edge  $\{x,y\}$. Then $Y=\{x,y\}$ and so there exists an ancestor of every element  of $Y$ in $N$ which is impossible.
By assumption on $Y$, it follows  for all $x\in Y$ that all elements in $Y-\{x\}$ have an 
ancestor $v_{Y,x}$ in $N$. Without loss of generality, we can choose  $v_{Y,x}$ 
such that no child of $v_{Y,x}$ also enjoys this property. 

We claim that the vertices $v_{Y,x}$, $x\in Y$, are pairwise incomparable and 
	therefore necessarily distinct. To see the claim, assume for contradiction that there exist $x,y\in Y$ 
	distinct such that $v_{Y,x}$ and $v_{Y,y}$ are not incomparable. Then $v_{Y,x}$ is an 
	ancestor of $v_{Y,y}$ or vice versa. Assume without loss of generality that $v_{Y,x}$ is 
	an ancestor of $v_{Y,y}$. Then $v_{Y,x}$ is an ancestor of all elements in $Y$ as 
	$x\in Y-y$ and $v_{Y,y}$ is an ancestor of the elements in $Y-y$; a contradiction in view of our assumption on $Y$.

Consider three distinct elements $x,y,z \in Y$ and the corresponding 
vertices $v_{Y,x},v_{Y,y},v_{Y,z} \in V(N)$. Since $v_{Y,x}$ and $v_{Y,y}$ are both ancestors 
of $z$ and incomparable, there exists a hybrid vertex $h_z$ that lies on the directed
paths from $v_{Y,x}$ to $z$ and from $v_{Y,y}$ to $z$. Note 
that we can choose $h_z$ such that no strict ancestor of $h_z$ belongs to 
those two paths. We can define vertices $h_y$ and $h_x$ in a similar way. 
It follows that the sequence $v_{Y,x},h_z,v_{Y,y},h_x,v_{Y,z},h_y$ is a 3-alternating cycle 
of $N$, which is impossible by assumption on $N$. Hence, all 
elements of $Y$ share an ancestor in $N$.
\end{proof}

Note that the assumption that $N$ does not contain a $3$-alternating cycle is necessary for
Lemma~\ref{lm-clique} to hold. In particular,
there exists networks $N$ that contain 3-alternating cycles and are such 
that for all $Y \subseteq X$, $|Y|\geq 2$, that is a clique in $\mathcal A(N)$ 
there exists a vertex $v$ in $N$ 
that is an ancestor of all leaves in $Y$. For example, the phylogenetic network $N$ depicted 
in Figure~\ref{fig-intro}(iii) contains a 3-alternating cycle, $\mathcal A(N)$ is a 
	clique with vertex set $Y=\{1,2,3\}$, and  $C(v)=Y$. However if we remove 
$v$ and its incident arcs from $N$ (suppressing resulting vertices of indegree and outdegree 1), 
then no vertex of the resulting network is an ancestor of all elements in $Y$.

We now continue with finding a network that represents a connected graph $G$
with vertex set $X$ 
with a minimum number of roots. To this end, we say that a subset $K$ of $\mathcal P(X)$ is 
an \emph{edge clique cover} of $G$ if every $Y \in K$ is a (not necessarily maximal) clique in $G$, 
and for every edge $\{x,y\}$ in $G$, there exists $Y \in K$ such that $x,y \in Y$. 
Note that $K(G)$ is always an edge clique cover of $G$, although $G$ may admit edge clique covers 
containing fewer elements than $K(G)$. We define the \emph{edge clique cover number} 
$\mathrm{ecc}(G)$ of $G$ as $\mathrm{min}\{|K| \,:\, K \text{ is an edge clique cover of } G\}$. In other 
words, $\mathrm{ecc}(G)$ is the minimum size of an edge clique cover of $G$ over all
such covers. 

Interestingly, and as Lemma~\ref{lm-minr} shows, the edge clique cover number of 
a connected graph $G$ 
provides a lower bound on the number of roots of a network that represents $G$.

\begin{lemma}\label{lm-minr}
Let $G$ be a connected graph with vertex set $X$. For $N$ a network on $X$ 
representing $G$, we have $ecc(G) \leq r(N)$.
\end{lemma}

\begin{proof}
It suffices to show that the set $K=\{C(r) \,:\, r \in R(N)\}$ is an edge clique cover of $G$. 
Clearly, every set $C(r)$, $r \in R(N)$, is a clique of $G$. Now, suppose for contradiction 
that $K$ is not an edge clique cover of $G$. Then there exists an edge $\{x,y\}$ of $G$ 
such that no root $r$ of $N$ satisfies $x,y \in C(r)$. In particular, no vertex $v$ of $N$ satisfies $x,y \in C(v)$. 
But this is impossible since $N$ represents $G$. The lemma now follows since, clearly,  $\mathrm{ecc}(G) \leq |K| \leq r(N)$.
\end{proof}

To prove the main result of this section (Theorem~\ref{thm-unirep}), we require some further definitions.
First, given  a set $\mathcal C \subseteq \mathcal P(X)$ of non-empty subsets of $X$,
we define a network $N(\mathcal C)$ on $X$ as follows. First take the 
cover digraph $H(\mathcal C)$ of $\mathcal C$ \cite[p.252]{S16},
that is, the digraph with vertex set $\mathcal C$, and
two distinct vertices $A, B \in \mathcal C$ joined by the arc $(A,B)$ if and only 
if $B \subsetneq A$ and there is no set $C \in \mathcal C$ with $B \subsetneq C \subsetneq A$.
To obtain $N(\mathcal C)$ from $H(\mathcal C)$, we first add (i) for all $x \in X$ 
with $\{x\} \notin \mathcal C$, a new vertex $\{x\}$ with outdegree $0$, and 
(ii) an arc from a vertex $A$ in $H(\mathcal C)$ to the vertex $\{x\}$ if $x \in A$ 
and no child of $A$ in $H(\mathcal C)$ contains $x$. To the resulting digraph we then 
(i) add  a child to every vertex with outdegree $0$ and indegree $2$ or more, and (ii) 
identify all leaves $l$ in the resulting digraph
with the unique element $x \in X$ such that $l=\{x\}$ or $l$ is a child of $\{x\}$.

Now, for any connected graph $G$ with vertex set $X$, and any edge clique cover $K$ of $G$,
we let 
$$
\mathcal C(K) = \{\bigcap_{Y \in S} Y \,:\, S \subseteq K \mbox{ and } \bigcap_{Y \in S} Y \neq \emptyset \},
$$
and we set $N(K)=N(\mathcal C(K))$.
As an illustration of these definitions, consider the graph $G$ depicted in Figure~\ref{fig-nk}(i). 
Then $N(K)$ is pictured in Figures~\ref{fig-nk}(ii) and \ref{fig-nk}(iii)
for $K$ the edge clique cover  $\{ \{1,2,3,4\}, \{3,4,5,6\}\}$ and $\{ \{1,2,3\}, \{1,2,4\}, \{3,4\}, \{3,5,6\}, \{4,5,6\}\}$ 
of $G$, respectively.

\begin{figure}[h]
	\begin{center}
		\includegraphics[scale=0.8]{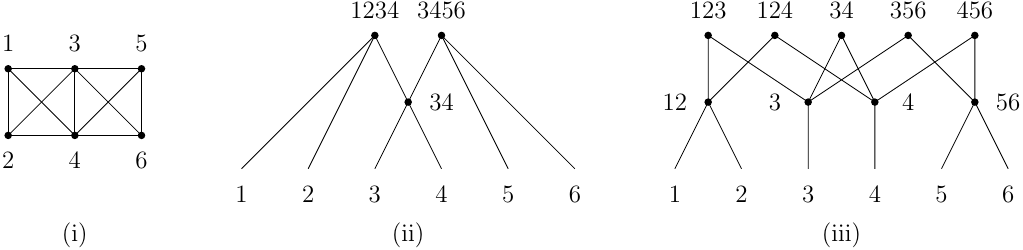}
	\end{center}
	\caption{(i) A graph $G$ with vertex set $X=\{1, \ldots, 6\}$.
			(ii) The network $N(K)$ for the edge-clique 
			cover $K=\{ \{1,2,3,4\}, \{3,4,5,6\}\}$ of $G$. 
			(iii) The network $N(K)$ for the edge-clique cover 
			$K=\{ \{1,2,3\}, \{1,2,4\},$ $ \{3,4\}, \{3,5,6\}, \{4,5,6\}\}$ of $G$. 
			For brevity, we represent a vertex $\{a_1,\ldots, a_p\}$, $p\geq 1$, of $ N(K)$ as the string $a_1a_2\ldots, a_p$.
	\label{fig-nk}
}
\end{figure}

We now show how an edge clique cover of a connected graph $G$ gives rise to a network representing $G$.

\begin{theorem}\label{thm-unirep}
Suppose that $G$ is a connected  graph with vertex set $X$. If $K$ is an edge clique cover of $G$, then
$N(K)$ is a network on $X$ that represents $G$. Moreover, $R(N(K)) \subseteq K$, and $R(N(K))=K$ if and only if $K$ 
does not contain two distinct sets such that one is a subset of the other. 
In particular, if $|K|$ is minimum (so that $|K|=ecc(G)$ and $R(N(K))=K$), 
then $N(K)$ has a minimum number of roots amongst all representations of $G$.
\end{theorem}

\begin{proof}
To ease notation, we put $N=N(K)$.

We first show that $N$ is a network on $X$.
Clearly, $N$ is acyclic and directed by definition. By construction, 
all vertices in $N$ with outdegree $0$ have indegree $1$, and so $X$ is contained in the
leaf set of $N$. To see that  the leaf set of $N$ is also contained in $X$, 
suppose that $N$ has a leaf $l$ that is not in $X$. 
Then $l$ corresponds to a set $A$ of $\mathcal C(K)$ of size $2$ or more. 
But by construction, for all $x \in A$, the vertex $x$ is a descendant of $l$, a contradiction. 
Note that this observation also shows that all sets $A \in \mathcal C(K)$ of size $2$ or more 
have at least two children in $N$. Hence, no vertex of $N$ has 
indegree and outdegree 1 in $N$ and all roots of $N$ have outdegree at least $2$.

To see that $N$ is a network, it remains to show that $N$ is connected. 
Suppose $x,y \in X$ distinct. Since $G$ is connected, there is a path
$x=v_1, \ldots v_k=y$, $k \geq 2$, in $G$, such that $v_i \in X$, $1 \le i \le k$. 
Since $K$ is an edge clique cover of $G$, for every such $i$, there exists a 
set $Y_i \in K$ such that $v_i, v_{i+1} \in Y_i$. In particular, $Y_i$ 
is a vertex of $N$ since $\{Y_i\}\subseteq K$, 
and there exists directed paths from $Y_i$ to $v_i$ 
and from $Y_i$ to $v_{i+1}$ in $N$. Hence, for all $1 \leq i \leq k-1$, 
there exists an path between $v_i$ and $v_{i+1}$ in the underlying  graph $U(N)$ of $N$. 
So there is a path in $U(N)$ between $x$ and $y$. Since this holds for 
all $x, y \in X$ and $N$  is acyclic, it follows that  $U(N)$ is connected.
Hence, $N$ is connected.

To see that $N$ is a representation of $G$, suppose that 
$x, y \in X$ distinct. Then, by construction, $x$ and $y$ share an ancestor in $N$ if and only 
if there exists some $Y \in K$ such that $x,y \in Y$. Since $K$ is an edge clique cover of $G$, this is the case 
if and only if $\{x,y\}$ is an edge in $G$, as required. 

To see that $R(N)\subseteq K$, note that for all $Y \in K$, we have $Y \in V(N)$
because $\{Y\}\subseteq K$. 
Moreover, all vertices $Z \in V(N)$ satisfy $Z \subseteq Y$ for some $Y \in K$. 
In particular, if $Z$ has indegree $0$ in $N$, then $Z \in K$. Hence, $Z$ must be a roof of $N$ and so $R(N) \subseteq K$. 

To see that $R(N)= K$ holds under the stated condition, note that a set $Z \in K$ 
has indegree $0$ in $N$ if and only if $Z \in K$ and 
no element $Z' \in K$ satisfies $Z \subsetneq Z'$. Hence, $R(N)=K$ holds if and only if $K$ 
does not contain two distinct sets such that one is a subset of the other.

Using this last observation, to see that the final statement of the theorem holds, it suffices to remark 
that if $|K|=\mathrm{ecc}(G)$, then $K$ does not contain $Y,Y'$ 
such that $Y \subsetneq Y'$. Otherwise, $K-\{Y\}$ is an edge clique cover of $G$ that
contains strictly fewer elements than $K$, a contradiction. So, in view of the above, it 
follows that $r(N)=|K|=\mathrm{ecc}(G)$. By Lemma~\ref{lm-minr}, $\mathrm{ecc}(G) \leq r(N')$ 
holds for all representations $N'$ of $G$, so the theorem follows.
\end{proof}

\section{Ptolemaic graphs}\label{sec-ptol}

In this section, we present some properties of Ptolemaic graphs, as defined in the introduction.
We begin by stating two key characterizations of Ptolemaic graphs from the literature.

For $k \geq 3$, we let $C_k$ denote the cycle on $k\geq 3$ vertices.
A graph $G$ is \emph{chordal} if it contains no induced cycle of
length 4 or more. In addition, the {\em gem}
is the graph pictured in Figure~\ref{PCk}.
In the following result, the equivalence between (i) and (ii) is proven in \cite{H81}, and the equivalence between 
(i) and (iii) is proven in \cite[Theorem 5]{UU09}
\footnote{Note that the statement of Theorem~\ref{thm-uu}
is slightly more general than that of \cite[Theorem 5]{UU09}
since in \cite{UU09} a Ptolemaic graph is assumed to be connected.}.

\begin{figure}[h]
	\begin{center}
		\includegraphics[scale=0.8]{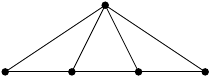}
	\end{center}
	\caption{The gem.}
	\label{PCk}
\end{figure}

\begin{theorem}\label{thm-uu} 
	Suppose that $G$ is a graph. Then the following are equivalent
	\begin{itemize}
	\item[(i)] $G$ is Ptolemaic.
	\item[(ii)] $G$ is chordal and does not contain the gem as an induced subgraph.
	\item[(iii)] the underlying  graph of $H(K(G))$ is acyclic. 
	\end{itemize}
\end{theorem}

We now make a general observation concerning the edge clique cover number of a Ptolemaic graph.

\begin{theorem}\label{lm-prime}
	Let $G$ be a connected graph with vertex set $X$. If $G$ is Ptolemaic, then there is no edge
		clique cover $K$ of $G$ distinct from $K(G)$ such that $|K| \leq |K(G)|$.
	In particular, $\mathrm{ecc}(G)=|K(G)|$.
\end{theorem}

\begin{proof}
	Note first that we may assume that $G$ is not an isolated edge as otherwise the theorem trivially holds.
	Suppose for contradiction that there exists an edge clique cover $K$ of $G$ 
		distinct from $K(G)$ such that $|K|\leq |K(G)|$. Without loss of generality, we 
		may assume that $K$ has minimum size. 	For all $Y \in K$, pick some maximal 
	clique $m(Y)$ in $K(G)$ (which may be $Y$ itself) 
	such that $m(Y)$ contains $Y$. Note that
	we can clearly always pick some such clique $m(Y)$.
	Then the set $\mathcal M(K)=\{m(Y)\,:\, Y \in K\}\subseteq K(G)$  is an 
	edge clique cover of $G$, and we have $|\mathcal M(K)| \leq |K|$. Since, 
	by assumption, $|K|=\mathrm{ecc}(G)$, it follows that $|\mathcal M(K)| = |K|=\mathrm{ecc(G)}$.

	We claim that there exists $Y_0 \in K(G)$ and $x \in Y_0$ such 
	that the set $K_{(Y_0,x)}$ obtained from $K(G)$ by replacing $Y_0$ with $Y_0-\{x\}$ 
	in case $|Y_0|>2$, or removing $Y_0$ from $K(G)$ in case $|Y_0|=2$, is an 
	edge clique cover of $G$. To see this, we distinguish between the cases that $|K|=|K(G)|$ and that $|K|<|K(G)|$.
	
	If $|K|=|K(G)|$, 
	then $\mathcal M(K)=K(G)$ as $\mathcal M(K)\subseteq K(G)$ and  $|\mathcal M(K)| = |K|$. Since
	$K \neq K(G)$ by assumption, 
	there exists $Y_0 \in K(G)$ such that $Y_0 \notin K$. In view of $\mathcal M(K)=K(G)$ it follows that
	$Y_0$ is of the form $m(Y)$ for some $Y \in K$. In particular, $Y \subsetneq Y_0$ holds. 
	Choose some $x\in Y_0-Y$. Then the definition of $K_{Y_0,x}$ implies that
	all sets of $K_{(Y_0,x)}$ are supersets of some set in $K$. Hence, $K_{(Y_0,x)}$ 
	is an edge clique cover of $G$.
	
	If $|K|<|K(G)|$, then $\mathcal M(K)$ is a proper 
	subset of $K(G)$. So for all $Y_0 \in K(G)-\mathcal M(K)$ and all $x \in Y_0$, 
	the set $K_{(Y_0,x)}$ contains $\mathcal M(K)$. Since $\mathcal M(K)$ is an edge 
	clique cover of $G$, it follows that $K_{(Y_0,x)}$ is also such a cover. This completes the proof of the claim.
	
	We next show that $G$ contains a $C_4$ or a gem. To this end, suppose that $Y_0 \in K(G)$ and $x \in Y_0$ are such that $K_0=K_{(Y_0,x)}$ is an 
	edge clique cover of $G$. 
	Let $Y_1 \in K_0$ such that $Y_1 \cap Y_0$ contains at least two 
	elements one of which is $x$. 
	Note that such a set $Y_1$ always exists since $Y_0$ is a clique in $G$ and $K_0$ is an edge clique cover of $G$.
    Without loss of generality, we may assume that $Y_1$ is such that no $Y' \in K$ 
	distinct from $Y_1$ satisfies $Y_1 \cap Y_0 \subset Y' \cap Y_0$. Since 
	$Y_0\in K(G)$, we have, $Y_1 \cap Y_0 \neq Y_0$. Hence, there exists $z \in Y_0$ such that $z \notin Y_1$.
	Since $x \in Y_1$ and $z \notin Y_1$, we have $z \neq x$. Furthermore, since
	$x,z \in Y_0$ and $Y_0$ is a clique in $G$ it follows that $\{x,z\}$ is an edge in $G$. 
	Hence, since $K_0$ is an edge clique cover of $G$, there exists $Y_2 \in K_0$ such that $x,z \in Y_2$.
	Moreover, by the choice of $Y_1$,  there exists $y \in Y_1 \cap Y_0$ such that $y \notin Y_2$, since  
	otherwise, $Y_1 \cap Y_0 \subset Y_2 \cap Y_0$.
	
	Consider now an element $u \in Y_1$ such that $\{u,z\}$ is not an edge of $G$.
	Note that such an element always exists, since $z \notin Y_1$ together with the maximality 
	of $Y_1$ implies that $Y_1 \cup \{z\}$ cannot be a clique in $G$. Note also that 
	since $y,z \in Y_0$, we have that $\{y,z\}$ is an edge in $G$ because $Y_0$ is a clique in $G$. Hence, $u \neq y$. 
	Since $\{x,z\}$ is an edge in $G$, we have  $u \neq x$.
	Similarly, there exists $v \in Y_2$ such that $\{v,y\}$ is not an edge of $G$.
	Note that $v \neq u, z$ since if $v=z$ then $\{v,y\}$ is an edge in $G$
	as $z,y\in Y_0$ and $Y_0$ is a clique in $G$, and if $v=u$ then $\{v,y\}$ is an edge in $G$
as $u,y\in Y_1$ and $Y_1$ is a clique in $G$.

	Now, if $\{u,v\}$ is an edge in $G$, then the set $\{u,y,z,v\}$ is a $C_4$ in $G$, since 
	$y,u \in Y_1$, $y,z\in Y_0$, and $z,v \in Y_2$ imply that $\{y,u\}$, $\{y,z\}$, and $\{z,v\}$ are edges in $G$
	as $Y_1$, $Y_0$,  and $Y_2$ are cliques in $G$, respectively. But then $G$ is not chordal since, as shown above, neither $\{u,z\}$ not $\{y,v\}$ can be an edge in $G$.
	Otherwise, the set $\{u,y,z,v,x\}$ induces a gem in $G$
	since $\{y,z\}$ is an edge in $G$ and $x,y \in Y_0$, $x,u \in Y_1$, and $x,v, z \in Y_2$ imply that  $\{x,y\}$, $\{x,u\}$, $\{v,z\}$, and $\{x,v\}$ 
	are also edges in $G$. In either case, it follows by Theorem~\ref{thm-uu} that $G$ is not Ptolemaic, a contradiction.
\end{proof}

Note that the converse of Theorem~\ref{lm-prime} is not true in general, that is, there exists graphs $G$ that 
are not Ptolemaic and are such that $K(G)$ is the only minimum size edge clique cover of $G$. 
This is the case, for example, if $G$ is isomorphic to $C_k$, $k \geq 4$.

\section{Arboreal representations}
\label{sc-abrep}

In this section, we characterize \emph{arboreal-representable} graphs, that is, 
graphs $G$ for which there exists an 
arboreal network $N$ on $X$ that represents $G$. We begin by considering some properties 
of the shared ancestry graph of an arboreal network.

\begin{lemma}\label{lm-tri}
Let $N$ be an arboreal network. Then:
\begin{itemize}
\item[(i)] if $N$ contains a non-root vertex of outdegree $2$ or more, then $\mathcal A(N)$ contains a $C_3$.
\item[(ii)] if $N$ has a vertex of outdegree $3$ or more, then $\mathcal A(N)$ contains a $C_3$.
\end{itemize}
\end{lemma}

\begin{proof}
To help establish Assertions~(i) and (ii), we first make the following claim. If 
$v$ is a vertex of $N$ with outdegree $k \geq 2$ then  $|C(v)|\geq k$. To see this, let $v$ be such a vertex.
Let $w$ and $w'$ be two distinct children of $v$. 
If $C(w) \cap C(w') \neq \emptyset$, then there exists a hybrid vertex $h$ of $N$ that is 
a descendant of both $w$ and $w'$. Assuming without loss of generality that 
no strict ancestor of $h$ also enjoys this property, it 
follows that $v,h$ is a 
1-alternating cycle in $N$. By Proposition~\ref{pr-arbalt}, this is impossible since $N$ is arboreal.
Hence, $C(w) \cap C(w') = \emptyset$ holds for any two distinct children $w,w'$ of $v$. 
Since, by assumption, $outdeg(v)\geq k$ the claim follows.

(i) Suppose that $v\in V(N)$  but not a root. Let $r$ be a root of $N$ that is an ancestor of $v$. 
By the previous claim, $|C(v)|\geq 2$. The same reasoning also implies that, there 
is an element $x \in X$ that is a descendant of $r$ but not of $v$. Since $C(v) \subseteq C(r)$, it 
follows that $|C(r)|\geq 3$. Since $C(r)$ is a clique in $\mathcal A(N)$ it follows 
that $\mathcal A(N)$ contains a $C_3$.

(ii) If $v$ has outdegree $3$ or more, then by the previous claim, $C(v)$ contains at 
least three elements. Since $C(v)$ is a clique in $\mathcal A(N)$ it follows that $\mathcal A(N)$ contains a $C_3$.
\end{proof}

\begin{lemma}\label{lm-acyclic}
Let $N$ be an arboreal network. Then $\mathcal A(N)$ is acyclic 
if and only if all vertices of $N$ have outdegree at most $2$, and the only vertices of $N$ 
with outdegree $2$ are the roots of $N$.
\end{lemma}

\begin{proof}
Assume first that $\mathcal A(N)$ is acyclic. In particular, $\mathcal A(N)$ 
does not contain a $C_3$. 
By Lemma~\ref{lm-tri}, it follows that all vertices in $N$ have outdegree at most $2$, 
and the only vertices in $N$ with degree $2$ are the roots of $N$.

Conversely, assume that all vertices in $N$ have outdegree at most $2$, and 
the only vertices in $N$ with outdegree $2$ are the roots of $N$. 
Then a vertex in $N$ must either be a root, a hybrid vertex, or a leaf. Since $N$ 
	is arboreal and so cannot contain a root $r$ and  some $x\in C(r)$ such that 
	there exists a directed path from $r$ to $x$ that contains two hybrid vertices of $N$, 
	it follows that $|C(r)|=2$. Hence, there exists a bijection between the roots of $N$ 
and the edges of $G$. Assume now for contradiction that $G$ 
contains a cycle $x_1, \ldots, x_k, x_{k+1}=x_1$, $k\geq 2$. Then for all $1 \leq i \leq k$, there 
exists a root $r_i$ in $N$ such that $C(r_i)=\{x_i,x_{i+1}\}$ in view of the aforementioned bijection. In particular, 
for all $1\leq i\leq k$ there 
exists a hybrid vertex $h_i$ that is common to the directed path from $r_i$ to $x_i$ 
and the directed path from $r_{i+1}$ to $x_i$. Without loss of generality, we may assume 
that no strict ancestor of $h_i$ belongs to both these paths. Hence, 
the sequence $r_1, h_1, r_2, \ldots, r_k,h_k$ is a $k$-alternating cycle in $N$.
By Proposition~\ref{pr-arbalt}, this is impossible since $N$ is arboreal. 
Hence, $\mathcal A(N)$ is acyclic as claimed.
\end{proof}

We now use Lemmas~\ref{lm-tri} and \ref{lm-acyclic} to relate the shared 
ancestry graph of an arboreal network with the Ptolemaic property. 
To help with this, we require a further concept.
For $N$ a network on $X$ and $Y$ a proper subset of $X$ with $|Y|\geq 2$, we define the 
restriction of $N$ to $Y$ to be the network $N'$ obtained from $N$ 
by first removing all leaves in $X-Y$ and their pendant arcs, then successively removing 
resulting vertices of outdegree $0$ (and their incoming arcs) and vertices of indegree $0$ 
and outdegree $1$ (and their outgoing arcs), and, finally, suppressing vertices of indegree and 
outdegree $1$, until no such vertices remain. For example, 
the restriction of the network depicted in Figure~\ref{fig-intro}(ii) to $Y=\{1,2,3,4\}$
is a rooted tree in which the arcs containing $1$ and $2$ share a vertex
and also the arcs containing $3$ and $4$.

We now show that the shared ancestry graph of an arboreal network is Ptolemaic.

\begin{proposition}\label{pr-ptol}
If $N$ is an arboreal network, then $\mathcal A(N)$ is Ptolemaic.
\end{proposition}

\begin{proof}
We first show that $G=\mathcal A(N)$ is chordal. Suppose 
for contradiction that $G$ contains an induced cycle $x_1, \ldots, x_k, x_{k+1}=x_1$, $k \geq 4$.
Let $Y=\{x_1, \ldots, x_k\}$, and let $N'$ be the restriction of $N$ to $Y$. 
Clearly, since $N$ is arboreal, $N'$ is arboreal. By definition, $\mathcal A(N')=G[Y]$ also holds.
So, by Lemma~\ref{lm-acyclic}, $N'$ must 
contain a vertex  with outdegree at least $2$ that is not a root, or 
one of the roots of $N'$ has outdegree $3$ or more. In both cases, 
it follows by Lemma~\ref{lm-tri} that $G[Y]$ contains a $C_3$, 
which contradicts the assumption that $Y$ induces a cycle in $G$ with 
length at least 4. Thus, $G$ is chordal.

Using Theorem~\ref{thm-uu} to complete the proof, we next show that $G$ 
does not contain a gem as an induced subgraph. To this
end, assume for contradiction that there exists a subset 
$Y=\{x,y,z,u,v\} \subseteq X$ such that $G[Y]$ is a gem. Let $N'$ be 
the restriction of $N$ to $Y$. Then similar arguments as before imply that $N'$ 
is arboreal and that $\mathcal A(N')=G[Y]$. 
Up to permutation in $Y$, we may assume that the edges of $G[Y]$ 
are $\{u,y\}$, $\{y,z\}$, $\{z,v\}$, $\{x,u\}$, $\{x,y\}$, $\{x,z\}$ and $\{x,v\}$.

Since, by definition, $N'$ represents $G[Y]$, it follows that $N'$ contains a root $r_1$ that is an 
ancestor of $u$ and $y$, a root $r_2$ that is an ancestor of $y$ and $z$, and a root $r_3$ 
that is an ancestor of $z$ and $v$. Note that since neither $\{u,z\}$ nor $\{v,y\}$ are 
edges of $G[Y]$, the roots $r_1$, $r_2$ and $r_3$ are pairwise distinct. In particular, 
there exists a hybrid vertex $h_y$ (\emph{resp.} $h_z$) in $N'$ that is a 
descendant of both $r_1$ and $r_2$ (\emph{resp.} $r_2$ and $r_3$), and 
no strict ancestor of $h_y$ (\emph{resp.} $h_z$) enjoys this property. 
Note that $h_y$ and $h_z$ are incomparable in $N'$. 

Now, since $N'$ is arboreal, the underlying undirected graph of $N'$ is a tree. 
Suppressing all vertices in this tree with 
degree 2, results in a tree $T$ with leaf set $\{x,y,z,u,v\}$
which either (i) has a single internal vertex with degree 5, 
(ii) two internal vertices, one with degree 3 and one with degree 4, or 
(iii) three internal vertices each with degree 3.

We now show that each of these cases leads to a contradiction, which will 
complete the proof. Case (i) is impossible, since each hybrid vertex in $N'$ corresponds to an
internal vertex in $T$ (since in the tree underlying $N'$ it has degree at least 3), 
and there are at least two hybrid vertices in $N'$.
In Case (ii), each of the two internal vertices in $T$ with degree greater than 2 must correspond to hybrid 
vertex in $N'$ which, in particular, implies that one of the leaves adjacent 
to the internal degree 3 vertex in $T$ corresponds to a vertex with degree 1 in $\mathcal A(N')$, which 
is impossible as  $\mathcal A(N')$ is a gem. Finally, in Case (iii), note that 
at least one of the two vertices in $T$ with degree 3 that are adjacent to two leaves in $T$ must be a hybrid vertex
as there are at least two hybrid vertices in $N'$. But, as in Case (ii), this implies that there must be
a vertex of degree 1 in $\mathcal A(N')$ which is impossible. This completes the proof of
the proposition.
\end{proof}

We are now ready to characterize arboreal-representable graphs.

\begin{theorem}\label{thm-wide}
Let $G$ be a connected graph with vertex set $X$. The
following statements are equivalent:
\begin{itemize}
\item[(i)] $G$ is Ptolemaic.
\item[(ii)] The underlying graph of $H(K(G))$ is acyclic.
\item[(iii)] $N(K(G))$ is arboreal.
\item[(iv)] $G$ is arboreal representable.
 \item[(v)] $G$ is arboreal representable by a network with $ecc(G)=|K(G)|$ roots.
\end{itemize}
\end{theorem}

\begin{proof}
To ease notation, we put $N=N(K(G))$, $H=H(K(G))$, and 
	$\mathcal C=\mathcal C(K(G))$. Note that the equivalence of
(i) and (ii) holds by Theorem~\ref{thm-uu}.
We now show that (ii) and (iii) are equivalent.

Suppose first that (iii) holds, i.e. $N$ is arboreal. Since 
$N$ is constructed from $H$ by adding new arcs and vertices, it follows that 
$H^+$ is a subgraph of $N$,. Hence, the underlying 
 graph of $H$ is acyclic. Thus, (ii) holds.

Conversely, suppose that (ii) holds, i.e. the underlying graph of $H$ is acyclic. 
Let $H^+$ be the graph obtained within the construction of $N$ from $H$ by adding, 
for all $x \in X$ such that $\{x\} \notin \mathcal C$, a new vertex $\{x\}$ 
with outdegree $0$ and with parents all the sets $A \in \mathcal C$ 
that contain $x$ and are such that no child of $A$ in $H$ contains $x$. 
This operation creates a cycle in the underlying graph 
of $H^+$ if and only if $H$ has two or more vertices $A$ and $B$
containing $x$ such that no child of $A$ in $H$ and no child of $B$ in $H$ contains $x$. We 
claim that this cannot be the case.

Indeed, suppose for contradiction that $\mathcal C$ contains two elements $A,B$ 
such that $A$ and $B$ contain $x$, and no child of $A$ and no child of $B$ in $H$ 
contains $x$. Since $A,B \in \mathcal C$, their choice implies that
there exists $S_A, S_B \subseteq K(G)$ distinct 
such that $A=\bigcap_{Y \in S_A} Y$ and $B=\bigcap_{Y \in S_B} Y$. 
In particular, we have $A \cap B=\bigcap_{Y \in S_A \cup S_B} Y$. Since $S_A \cup S_B \subseteq K(G)$, $A \cap B \in \mathcal C=V(H)$ follows by definition of $\mathcal C$. 
By definition of $H$, $A \cap B$ is a descendant of $A$ and $B$ in $H$. Since 
$x \in A \cap B$, we obtain a contradiction. This completes the proof of the claim.

It follows that the underlying graph of $H^+$ is acyclic. 
Since $N$ is obtained from $H^+$ by adding a 
new child to each vertex of $H^+$ with outdegree $0$ and indegree $2$ or more,
this operation does not create a cycle in the underlying  graph of $H^+$. Hence
$N$ must be arboreal, i.e. (iii) holds.	

To complete the proof, first note that
$N$ represents $G$ by 
Theorem~\ref{thm-unirep} as $K(G)$ is an edge clique cover of $G$. Hence, (iii) implies (v) in view of Theorem~\ref{thm-unirep} and 
Theorem~\ref{lm-prime} since $\mathcal A(N)$ is Ptolemaic by Proposition~\ref{pr-ptol}. 
Moreover, that (v) implies (iv) is trivial, and that (iv) implies (i) 
follows immediately from Proposition~\ref{pr-ptol}.
\end{proof}

\section{Symbolic maps}
\label{sc-sym}

In this section, we characterize symbolic arboreal maps.
We begin by considering properties of ancestors in networks.

Let $N$ be a network on $X$. For $x, y \in X$ two distinct leaves of $N$, 
we say that $v \in V(N)$ is a \emph{least common ancestor} of $x$ and $y$ if $v$ is an ancestor of both $x$ and $y$, 
and no child of $v$ in $N$ enjoys this property. It is well-known that if $N$ is a 
phylogenetic tree, then any  two leaves of $N$ have a unique least common ancestor. 
As we have seen in Section~\ref{sec-sag}, in networks, two leaves do 
not necessarily have a least common ancestor. It is therefore of interest to understand when the uniqueness property holds for 
leaves that share an ancestor. The next result shows that this is always the case for arboreal networks.

\begin{proposition}\label{pr-lcau}
Let $N$ be a network on $X$. If $N$ does not contain a $2$-alternating cycle, 
then if $x,y\in X$ share an ancestor in $N$, then $x$ and $y$ have a unique least common ancestor in $N$. 
In particular, if 
$N$ is an arboreal networks, then the least common ancestor of two leaves sharing an ancestor is unique.
\end{proposition}

\begin{proof}
	Let $N$ be an arboreal network on $X$ that does not contain a 2-alternating cycle.
Let $x,y \in X$ such that $x$ and $y$ share an ancestor in $N$. Then $x$ and $y$ clearly
have at least one least common ancestor in $N$. Assume for the following that $x\not=y$ 
	since otherwise the proposition trivially holds.

To see that there exists exactly one such vertex, assume for contradiction 
that there exists $v,w \in V(N)$ distinct such that both $v$ and $w$ enjoy 
the property that they are a least common ancestor of $x$ and $y$. Then
there exists two distinct children $v_x$ and $v_y$ of $v$ that are ancestors of $x$ and $y$ 
respectively, and two distinct children $w_x$ and $w_y$ of $w$ that are 
ancestors of $x$ and $y$, respectively. Since $v_x$ and $w_x$ are both 
ancestors of $x$
there must exist a hybrid vertex $h_x$ belonging 
to a directed path from $v_x$ to $x$ and a directed path from $w_x$ to $x$. Without loss of generality, 
we may choose $h_x$ such that no strict ancestor of $h_x$ enjoys this property. Clearly, $y$ 
is not a descendant of $h_x$ as otherwise $y$ is a descendant of $v_x$ and $w_x$ 
which contradicts the fact that $v$ and $w$ are least common ancestors of $x$ and $y$ in $N$. 
By symmetry, $v_y\not=w_y$. Hence,
there must also exist a vertex $h_y$ belonging to a directed path from $v_y$ to $y$ and a directed path from $w_y$ to $y$. 
Again, we may assume without loss of generality that no strict ancestor of $h_y$ enjoys
this property. Hence, $v,h_x,w,h_y$ is a $2$-alternating cycle in $N$, a contradiction. Thus,
$x$ and $y$ have a unique least common ancestor in $N$.
\end{proof}

Note that the converse of Proposition~\ref{pr-lcau} does not hold in general, 
since there exist networks $N$ on $X$ that contain $2$-alternating cycles, 
and are such that the least common ancestor of $x$ and $y$ is unique for all $x,y \in X$ 
that share an ancestor in $N$. For example, the phylogenetic network $N$ depicted in 
Figure~\ref{fig-intro}(iii) contains three 2-alternating cycles, but one can easily 
check that any pair of elements of $\{1,2,3\}$ has a unique least common ancestor in $N$. 

Assume for the rest of the paper that $M$ 
is a non-empty set and that $\odot\not\in M$. As in the introduction, we set $M^{\odot}= M \cup \{\odot\}$
and call a symmetric map $d: {X \choose 2} \to M^{\odot}$ a 
{\em symbolic map (on $X$)}.
Denoting for a network $N$ the set of all vertices 
with outdegree $2$ or more by $V(N)^-$,
we call a pair $(N,t)$ consisting of a network $N$ on $X$ 
and a map $t:V(N)^-\to M$ a \emph{labelled network (on $X$)}. In this case, we also call the
map $t$ a {\em labelling map (for $N$)}. 

For $N$ an arboreal network and $x,y$ two leaves of $N$ that share 
an ancestor, we denote by $\mathrm{lca}_N(x,y)$ the least common ancestor of $x$ and $y$ in $N$, 
which is well defined by Proposition~\ref{pr-lcau}. As mentioned in the 
introduction, every labelled arboreal network $(N,t)$ on $X$ induces a 
(unique) symbolic map $d_{(N,t)}:{X\choose 2}\to M^{\odot}$
which, for $\{x,y\} \in {X \choose 2}$, is defined by taking $d_{(N,t)}(x,y)=t(\mathrm{lca}_N(x,y))$ if $x$ and $y$
share an ancestor in $N$, and $d_{(N,t)}(x,y)= \odot$ else. We say that a labelled 
arboreal network $(N,t)$ on $X$ \emph{explains a symbolic map $d$ on $X$} if $d=d_{(N,t)}$, 
in which case, we call $d$ a \emph{symbolic arboreal map}. Note that these maps have a special
property in case $N$ is a phylogenetic tree:

\begin{lemma}\label{lm-dt}
	Let $(N,t)$ be a labelled arboreal network on $X$.
	Then $d_{(N,t)}(x,y) \not=\odot$ for all $\{x,y\} \in {X \choose 2}$ if and only if $N$ is a phylogenetic tree on $X$.
\end{lemma}

\begin{proof}
	Set $d=d_{(N,t)}$.
		Note that since $N$ is arboreal, it must be connected.
	
	Suppose first that $d(x,y) \in M$, for all $\{x,y\}\in {X\choose 2}$. Then any two 
	leaves of $N$ share an ancestor. 
	Thus, $X$ is a clique in $\mathcal A(N)$. Since $N$ is arboreal and so 
		cannot contain a 3-alternating cycle in view of Proposition~\ref{pr-arbalt}, it follows by
		Lemma~\ref{lm-clique}  that
	$N$ contains a vertex 
	$v$ that is an ancestor of all elements of $X$. Using Proposition~\ref{pr-arbalt} again, 
		it follows that, $N$ cannot contain a hybrid vertex. Hence,
	$v$ is necessarily the only root of $N$. Thus, $N$ is a phylogenetic tree on $X$.
	
	Conversely, suppose that $N$ is a phylogenetic tree on $N$. Then any two leaves of $N$ 
	share an ancestor, so $d(x,y) \in M$ for all $\{x,y\} \in {X \choose 2}$.
\end{proof}
 
Now, suppose that $d$ is a symbolic map on $X$.
Let $G_{d}$ be the graph with vertex set $X$, such that $\{x,y\} \in {X\choose 2}$
are joined by an edge if and only if $d(x,y) \neq \odot$.  We next present a key link between 
	the graph $G_d$ associated to a symbolic map $d$ on $X$ and the shared ancestry graph of a network on $X$.

\begin{lemma}\label{lm-du}
Let $(N,t)$ be a labelled arboreal network on $X$.
Then $G_{d_{(N,t)}}$ and $\mathcal A(N)$ are isomorphic and that isomorphism is the identity on $X$.
\end{lemma}

\begin{proof}
Put $d=d_{(N,t)}$ and recall that $X$ is the vertex set of both $G_d$ and $\mathcal A(N)$. Let $x,y\in X$ distinct.
By definition, $\{x,y\}$ is an arc of $\mathcal A(N)$ if and only if $x$ and $y$ 
share an ancestor in $N$. Since, by definition, $N$ explains $d$, $x$ and $y$ 
share an ancestor in $N$ if and only if $d(x,y) \neq \odot$, that is, if and only if $\{x,y\}$ is an 
edge of $G_d$.
\end{proof}

Before presenting the main result of this section (Theorem~\ref{thm-diss}), we recall some facts concerning symbolic ultrametrics including the 3- and 4-point conditions stated in the introduction.
Suppose that $d:{X\choose 2}\to M^{\odot}$ is a symbolic map.
We say that three pairwise distinct elements $x,y,z \in X$ are in \emph{$\Delta$-relation} (\emph{under $d$}) if $|\{d(x,y),$ $d(x,z),d(y,z)\}|=3$ 
and $\odot \notin \{d(x,y),d(x,z),$ $d(y,z)\}$. We also say that four pairwise 
	distinct elements $x,y,z,u \in X$ are in \emph{$\Pi$-relation} 
(\emph{under $d$}) if, up to permutation of the elements $x,y,z,u$, $d(x,y)=d(y,z)=d(z,u) \neq d(z,x)=d(x,u)=d(u,y)$ 
and $\odot \notin \{d(x,y),d(x,z)\}$. 
These relations naturally arise when explaining symbolic maps in terms of 
	phylogenetic trees (see e.\,g.,\cite{BD98,G84,G00}). Bearing in mind that every symbolic map $d: {X \choose 2} \to M^{\odot}$ 
		can be extended to a 
		map $d':X\times X\to (M\cup \{0\})^{\odot}$ by putting $d'(x,y)=d(x,y)$ if $x\not=y$ and 
		$d'(x,y)=0$ if $x=y$, Theorem 7.2.5 in \cite{SS03} implies:

\begin{theorem}\label{pr-deltapi}
Suppose that $d: {X \choose 2} \to M^{\odot}$ is a symbolic map. 
Then there exists a labelled phylogenetic tree $(T,t)$ on $X$ explaining $d$
	if and only if no three pairwise distinct
elements of $X$ are in $\Delta$-relation under $d$
and also no four pairwise distinct elements of $X$ are in  
$\Pi$-relation under $d$.
\end{theorem}

We now use this result to characterize symbolic maps that can be explained by a labelled arboreal network:

\begin{theorem}\label{thm-diss}
Suppose that $d: {X \choose 2} \to M^{\odot}$ is a symbolic map. Then, 
$d$ is a symbolic arboreal map if and only if the following four properties all hold:
\begin{itemize}
\item[(A1)] $G_d$ is connected and Ptolemaic.
\item[(A2)] No three pairwise distinct elements of $X$ are in $\Delta$-relation under $d$.
\item[(A3)] No four pairwise distinct elements of $X$ are in $\Pi$-relation under $d$.
\item[(A4)] If $x,y,z,u \in X$ are pairwise distinct and are such that $d(z,u)= \odot$ and $d$ maps all other elements of ${\{x,y,z,u\}\choose 2}$ to an element of $M$, then $d(x,z)=d(y,z)$ and $d(x,u)=d(y,u)$ hold.
\end{itemize}
\end{theorem}

\begin{proof} It is straight-forward to check that the theorem holds is
	$|X|\in\{2,3\}$ since  Properties~(A3) and (A4) 
	vacuously hold in case $|X|\leq 3$ and Property~(A2) vacuously holds in case $|X|=2$. So assume that $|X|\geq 4$ since
Suppose first that $d$ is a symbolic arboreal map, that is, there exists a labelled arboreal network $(N,t)$ 
explaining $d$. By Lemma~\ref{lm-du}, there exists an isomorphism between $G_d$ 
and $\mathcal A(N)$ that is the identity on $X$.
In particular, $G_d$ must be connected as $\mathcal A(N)$ is connected. Since, by 
Theorem~\ref{thm-wide}, $G_d$ is Ptolemaic  it follows that Property~(A1) holds.

We now show that Property~(A2) holds. As part of this, we remark that the 
	proof of Property~(A3) uses analogous arguments 
on subsets of $X$ of size 4.
Let $x,y,z$ be three pairwise distinct elements of $X$. If $\odot \in \{d(x,y),d(x,z),d(y,z)\}$, 
then since $(N,t)$ explains $d$, it follows that $x,y,z$ are not in $\Delta$-relation. 
So assume that $\odot \not\in \{d(x,y),d(x,z),d(y,z)\}$. Then
$\{x,y,z\}$ is a clique in $G_d$. 
By Lemma~\ref{lm-clique}, there exists a vertex $v$ in $N$ that 
is an ancestor of $x,y$ and $z$. Since $N$ is arboreal, it cannot contain 
a 3-alternating cycle by Proposition~\ref{pr-arbalt}. Let $T_v$ 
be the subtree of $N$ rooted at $v$. Note that $T_v$ must exist as $N$ is 
	arboreal and so cannot contain a $1$-alternating cycle by Proposition~\ref{lm-rh1}. For $t_v$ the 
restriction of $t$ to $V(T_v)$, it follows that the labelled phylogenetic tree
$(T_v,t_v)$ explains $d|_{L(T_v)}$. Property~(A2) then follows from Theorem~\ref{pr-deltapi}. 

To see that Property~(A4) holds, let $x,y,z,u \in X$ be pairwise distinct such that $d(z,u)= \odot$ and that all other elements in ${\{u,x,y,z\}\choose 2}$ are mapped to some element in 
$M$ under $d$. By Lemma~\ref{lm-clique}, there exists 
vertices $v$ and $w$ that are ancestors of the leaves in $\{x,y,z\}$ and $\{x,y,u\}$ 
respectively, and no vertex in $N$ is an ancestor of all four of $x,y,z,u$. 
In particular, $v$ and $w$ do not share an ancestor in $N$
as otherwise that ancestor would also be an ancestor of $u$ and $z$ which is impossible.
Since both $v$ and $w$ are ancestors of the leaves $x$ and $y$ in $N$, 
there exists a hybrid vertex $h_x$ that is common 
to the directed paths from $v$ to $x$ and from $w$ to $x$. 
Similarly, there exists a hybrid vertex $h_y$ that is common 
	to the directed paths from $w$ to $x$ and from $w$ to $y$.
Without loss of generality, we may assume that neither $h_x$ nor $h_y$ 
has an ancestor enjoying this property. 

We first remark that $h_x$ is an ancestor of $\mathrm{lca}_N(x,y)$. To see this, 
	it suffices to show that $h_x=h_y$. 
Assume for contradiction that $h_x \neq h_y$. 
By choice of $h_x$ and $h_y$, these two vertices are incomparable in $N$. 
Hence, $v,h_x,w,h_y$ is a $2$-alternating cycle in $N$, a contradiction 
in view of Proposition~\ref{pr-arbalt} as $N$ is arboreal. Thus, $h_x=h_y$ and, so,
$h_x$ is an ancestor of $\mathrm{lca}_N(x,y)$. 

Clearly, $h_x$ is not an 
ancestor of $z$, as otherwise $w$ is an ancestor of $z$. 
Similarly, $h_x$ is not an ancestor of $u$, as otherwise $v$ is an ancestor of $u$.
So we must have $\mathrm{lca}_N(x,z)=\mathrm{lca}_N(y,z)$ and $\mathrm{lca}_N(x,u)=\mathrm{lca}_N(y,u)$. 
Since $(N,t)$ explains $d$, it follows that $d(x,z)=d(y,z)$ and $d(x,u)=d(y,u)$ hold. 
This concludes the proof of Property~(A4).

Conversely, suppose that $d$ satisfies Properties~(A1)--(A4). We next construct a 
labelled arboreal network $(N,t)$ that explains $d$. To help illustrate our construction, 
	we refer the reader to Figure~\ref{work-ex} for an example. 

Since $G_d$ is connected and 
Ptolemaic, Theorem~\ref{thm-wide} implies that there exists an arboreal network $\widehat N$ 
on $X$ such that $\widehat N$ represents $G_d$. Without loss of generality, we may 
assume that $\widehat N$ does not contain an arc $(u,v)$ such that $u$ has outdegree $2$ 
or more and $v$ is a non-leaf tree-vertex, 
since contracting such arcs preserves $\mathcal A(\widehat N)$
(see Figure~\ref{work-ex}~(ii)) and so we could take the resulting network to be $\hat{N}$.
By construction, we have for any two distinct elements $x$ and $y$ in $X$ 
	that $x$ and $y$ share an ancestor in $\widehat N$ if and only if $d(x,y) \neq \odot$.

To obtain a labelled arboreal network from $\widehat{N}$ that explains $d$, let $v$ 
be a 
vertex of $\widehat N$ of outdegree $2$ or more. By assumption on $\widehat N$, the children 
of $v$ are either hybrid vertices of $\widehat N$ or leaves of $\widehat N$. We first 
claim that if $h$ is a child of $v$ that is a hybrid vertex, and $z$ is a descendant of $v$ 
that is not a descendant of $h$, then $d(x,z)=d(y,z)$ holds for all leaves $x,y$ below $h$ in $\widehat{N}$. To 
see this, let $x$ and $y$ be leaves of $\widehat N$ that are below $h$. Let
$v'$ be a tree vertex that is an ancestor of $h$ but not of $v$, and let $u$ be a 
leaf that is a descendant of $v'$ but not of $h$. Note that such a leaf must exist as $N$ is 
	arboreal and so cannot contain a 1-alternating cycle by Proposition~\ref{pr-arbalt}.
By choice of $x,y,z,u$, there is exactly one 
element in  ${\{x,y,z,u\}\choose 2}$ that is mapped to $\odot$ under $d$, that is, the element $\{z,u\}$.
By Property~(A4), $d(x,z)=d(y,z)$ holds, as claimed.

In view of this claim, we can ``locally replace" $v$ with a tree-structure as follows. 
Let $C_v$ be the set of children of $v$ in $\widehat N$. By assumption on $v$, we have $|C_v|\geq 2$.
For $v_1,v_2 \in C_v$ distinct, we define a symbolic map $d_v:{C_v\choose 2}\to M^{\odot}$ by putting
$d_v(v_1,v_2)=d(x_1,x_2)$ 
for some leaves $x_1$ and $x_2$ below $v_1$ and $v_2$, respectively. 
The fact that all non-leaf children of $v$ are hybrid vertices together with 
the previous claim imply that the definition of $d_v(v_1,v_2)$ 
does not depend on the choices of $x_1$ and $x_2$. Moreover, $d_v(v_1,v_2) \neq \odot$
for all $v_1,v_2 \in C_v$. Since Properties~(A2) and (A3) hold by the definition of $d_v$, it follows by 
Theorem~\ref{pr-deltapi} that there exists a labelled phylogenetic tree $(T_v,t_v)$ on $C_v$ that explains $d_v$ (see Figure~\ref{work-ex}(iii)). We can then modify $\widehat N$ at $v$ into an arboreal network $N_v$ on $X$
by (i) removing all outgoing arcs of $v$ in $\widehat{N}$, (ii) identifying $v$ 
with the root of $T_v$ and (iii) identifying each vertex $w \in C_v$ in 
with the corresponding leaf of $T_v$. Note that $N_v$ might be $\widehat{N}$. By construction, we have for all
leaves $x$ and $y$ below $v$ that $\mathrm{lca}_{N_v}(x,y)$ is a vertex of $T_v$ and that $t_v(\mathrm{lca}_{N_v}(x,y))=d(x,y)$.

Now, let $N$ be the network obtained by applying the above process to all non-leaf vertices 
of $\widehat N$ of outdegree $2$ or more (see Figure~\ref{work-ex}(iv)). By construction, for all vertices $w$ of $N$ of outdegree $2$ or more, 
there exists exactly one vertex $v$ of $\widehat N$ such that $w\in V(T_v)$. 
	Taken together, the maps $t_v$ induce a natural labelling map $t:V(N)^-\to M$.

It remains to show that $(N,t)$ explains $d$, that is, for all $\{x,y\}\in {X\choose 2}$ we have that  $d(x,y)=\odot$ if $x$ 
and $y$ do not share an ancestor in $N$, and $d(x,y)=t(\mathrm{lca}_{N}(x,y))$ otherwise.
To see this, let $x,y$ be two elements of $X$. If $d(x,y)=\odot$, then, as mentioned before, $x$ and $y$ 
do not share an ancestor in $\widehat N$. By construction, that property still holds in $N$. 
If $d(x,y) \neq \odot$, then $x$ and $y$ share an ancestor in $\widehat N$. 
Let $v$ be the least common ancestor of $x$ and $y$ in $\widehat N$. 
Then for $(T_v,t_v)$ the labelled phylogenetic tree obtained by replacing $v$ in the construction of $N_v$ from $\widehat{N}$, it 
follows in view of our observations concerning $N_v $ that $\mathrm{lca}_{N}(x,y)$ is a vertex of $T_v$
and that $t_v(\mathrm{lca}_{N}(x,y))=d(x,y)$. Since, by definition, $t(w)=t_v(w)$ for 
all internal vertices $w$ of $T_v$, we have $t(\mathrm{lca}_{N}(x,y))=d(x,y)$ as desired. 
Hence, $(N,t)$ explains $d$.
\end{proof}

\begin{figure}[h]
		\begin{center}
			\includegraphics[scale=0.7]{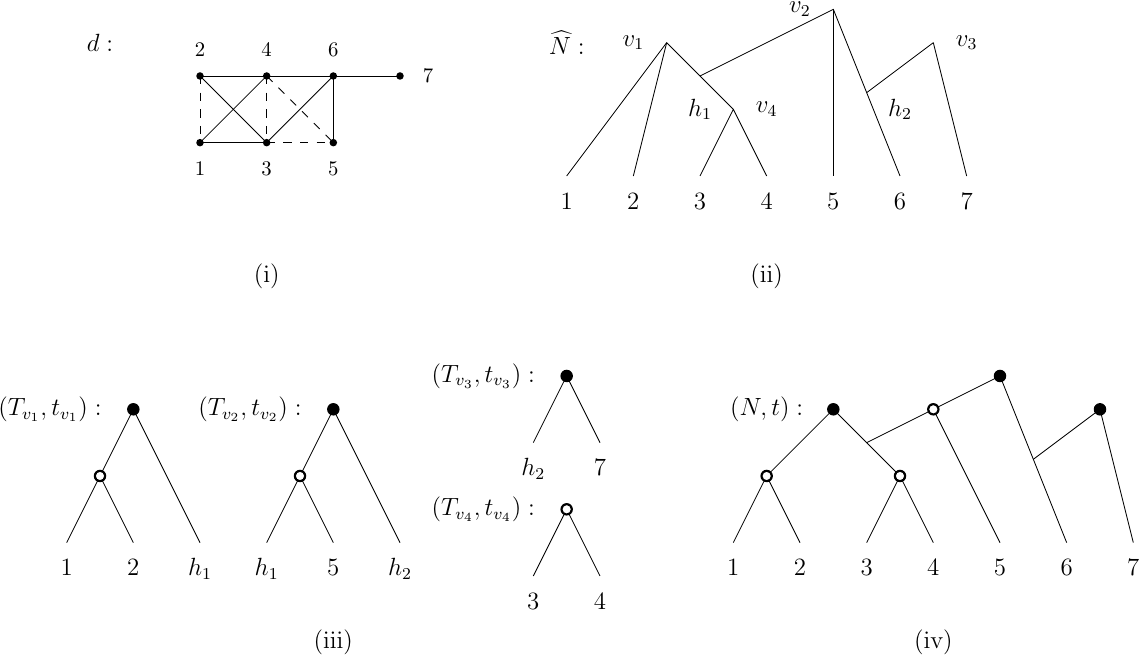}
		\end{center}
		\caption{(i) For  $X=\{1, \ldots, 7\}$, 
		a symbolic map $d:{X\choose 2}\to \{\bullet, \circ,\odot\}$ represented in terms of an edge labelled graph. For $x,y\in X$ distinct, there is
			an edge $\{x,y\}$ in that graph that is solid if $d(x,y)=\bullet$ and dashed if 
			$d(x,y)=\circ$. If
			there is no edge between $x$ and $y$ then $d(x,y)=\odot$. In particular, $G_d$ is the depicted graph, where the edge styles 
			are ignored. Using the notation from the proof of Theorem~\ref{thm-diss}, (ii) 
			presents the arboreal network $\widehat N$ for $G_d$ 
			in which no arc joins a vertex with outdegree $2$ or more with a non-leaf tree-vertex. (iii) For all internal tree-vertices $v_i$ of $\widehat N$, a labelled phylogenetic tree  $(T_{v_i},t_{v_i})$ on the set $C_{v_i}$ of children of $v_i$ that explains $d_{v_i}$. (iv) 
			The labelled arboreal network $(N,t)$ that explains $d$ obtained by replacing each 
			internal vertex $v_i$ of outdegree $2$ or more in $\widehat N$ by $(T_{v_i},t_{v_i})$.}
		\label{work-ex}
\end{figure}
	
We conclude this section by stating a uniqueness result. We say that two networks $N$ and  $N'$ on $X$ are \emph{isomorphic} if there exists a digraph isomorphism from $V(N)$ to $V(N')$ that is the identity on $X$. In \cite[Theorem 2]{BD98} it is shown that 
for any symbolic ultrametric $d$ there is a unique (up to isomorphism) labelled tree $(T,t)$ which explains $d$ 
which has the property that $t(u)\neq t(v)$ for any {\em internal arc} $(u,v)$ in $T$ (i.e. an arc that does not contain a leaf). 
In a similar vein, we say that a labelled arboreal network $(N,t)$ is \emph{discriminating} if $N$ 
has no internal arc $(u,v)$ such that $u$ has outdegree $1$, and no 
internal arc $(u,v)$ such that $v$ has indegree $1$ and $t(u)=t(v)$.
Then we have the following result:

\begin{theorem}\label{cor-discr}
	Let $d: {X \choose 2} \to M^{\odot}$ be a symbolic arboreal map. Then there exists a unique (up to isomorphism that is the identity on X) discriminating arboreal network $(N,t)$ on $X$ that explains $d$.
\end{theorem}

Note that if $N$ is a phylogenetic tree, then $N$ has no internal arc $(u,v)$ such that $u$ has outdegree $1$, so Theorem~\ref{cor-discr} is a generalization of the aforementioned uniqueness result for symbolic ultrametrics. As our proof for this result is somewhat long and technical we shall present it in the Appendix.
 
\section{Discussion}\label{sec-disc}

In this paper, we have characterized symbolic maps that can be explained by a labelled arboreal network.
To do this, we introduced the concept of the shared ancestry graph of a network, and 
then exploited the connection between such graphs and Ptolemaic graphs for arboreal networks.

It would be interesting to understand 
how far our results might be extended to other classes of networks or symbolic maps. 
For example, as mentioned in the introduction, results have recently appeared on connections between symbolic maps
and so-called level-1 phylogenetic networks \cite{HS18}, and so one might investigate if similar 
results can be derived in the setting where networks are permitted to have multiple roots.
In addition, there are connections between ultrametrics, edge-labelled hypergraphs and 
symbolic 3-way maps \cite{G00, HMS19} that might potentially yield interesting generalizations 
within the arboreal setting. And, finally, it could be worth investigating how
properties of symbolic arboreal maps vary with different choices of symbol set $M$; for example, 
in case $M$ is taken to be a group (see e.g. \cite{semple1999tree}).

In another direction, note that since a Ptolemaic graph can be recognized in linear time \cite{UU09}, as a corollary of 
Theorem~\ref{thm-diss} we immediately
obtain the following observation.

\begin{corollary}\label{cor:poly}
	A symbolic arboreal map on a set $X$ can be recognized in $O(|X|^4)$ time.
\end{corollary}

It would be interesting to
know if there is an algorithm for recognizing symbolic arboreal 
maps that has a better run-time than $O(|X|^4)$. 
Also for applications, it would be useful to develop an efficient algorithm for constructing a labelled arboreal 
network that explains a symbolic arboreal map. 
Such an algorithm is implicitly given in the proof of Theorem~\ref{thm-diss}, in 
which we describe the ``vertex-replacement" operation, which constructs a 
representation of $d$ from some $\widehat N$. For example, we can 
always choose $\widehat N$ to be $N(K(G_d))$, which we know how to construct from $K(G_d)$. 
Note that \cite[Theorem 8]{UU09} shows how to construct a 
directed clique laminar tree associated to a Ptolemaic graph in linear time 
might also be useful for developing algorithms for symbolic arboreal maps.

\bibliographystyle{siamplain}
\bibliography{bibliography}

\newpage 

\section{Appendix}

In this appendix, we prove Theorem~\ref{cor-discr}. To do this
we shall first consider properties of the sets $C(v)$ for $v$ a vertex of an arboreal network $N$, 
and then show that, for a labelled arboreal network $(N,t)$, we can recover the
sets $C(v)$ from the map $d_{(N,t)}$ which permits us to prove uniqueness. 
We begin with a result which  underlines the key role played  
by the elements in $\mathcal C(K(G))$ in case $G$ is the 
shared ancestry graph for an arboreal network $N$.

\begin{proposition}\label{pr-foset}
Let $N$ be an arboreal network and let $G=\mathcal A(N)$. For all $Z \in \mathcal C(K(G))$, there 
exists a vertex $v$ of $N$ such that $C(v)=Z$.
\end{proposition}

\begin{proof}
To ease notation, set $K=K(G)$.
Let $Z \in \mathcal C(K)$. The proposition holds if $|Z|=1$ since then $Z=C(x)$ for some $x\in X$. 
So assume for the remainder that $|Z|\geq 2$. We distinguish between the cases that $Z \in K$ and that $Z \notin K$.

Suppose first that $Z \in K$. Since $N$ is arboreal and so cannot contain a 3-alternating cycle by Proposition~\ref{pr-arbalt},  Lemma~\ref{lm-clique} implies that there exists a vertex $v_Z$ of $N$ such that $Z \subseteq C(v_Z)$. 
Let $x \in C(v_Z)$. Since $x$ and $z$ share an ancestor 
for all $z \in Z$, it follows that $Z \cup \{x\}$ is a clique in $G$. By maximality of $Z$
it follows that $x \in Z$. Hence, $C(v_Z) \subseteq Z$. Thus $C(v_Z) = Z$, which completes the proof of the proposition
in case $Z \in K$. 

So, suppose $Z \notin K$. Let $K_Z=\{Y \in K\,|\, Z \subset Y\}$. 
Note that since $Z \in \mathcal C(K)-K$, we have $|K_Z| \geq 2$ 
and $Z=\bigcap_{Y \in K_Z} Y$. By Lemma~\ref{lm-clique}, it follows 
that there exists a vertex $v_Z$ of $N$ such that $Z \subseteq C(v_Z)$. 
Without loss of generality, we can choose $v_Z$ such that no 
strict descendant of $v_Z$ satisfies this property.
We now show that $C(v_Z) \subseteq Z$ must also hold, which 
implies that $Z = C(v_Z)$ and thus completes the proof of the proposition.

We first claim that if $y \in X-C(v_Z)$ is such that $y$ and $z$ share an ancestor 
in $N$ for all elements $z \in Z$, then for all $x \in C(v_Z)$, $x$ and $y$ share an ancestor in $N$.

To see that the claim holds, suppose for contradiction that there 
exists $y \in X-C(v_Z)$ and $x \in C(v_Z)$ such that $y$ and $z$ share an 
ancestor in $N$ for all elements $z \in Z$ but $x$ and $y$ do not share an ancestor in $N$. 
By choice of $v_Z$, there exists two elements $z_1, z_2 \in Z$ distinct such that $z_1$ and $z_2$ 
are descendant of two distinct children $v_1$ and $v_2$ of $v_Z$, respectively. 
Indeed, if this is not the case, then all elements of $Z$ are descendant of the same child $v'$ of $v_Z$, 
which contradicts our choice of $v_Z$. 

Now, let $w_1=\mathrm{lca}(z_1,y)$ and $w_2=\mathrm{lca}(z_2,y)$. 
Since $x$ and $y$ do not share an ancestor in $N$, $v_Z$ is incomparable 
with $w_1$ and $w_2$. For $i \in \{1,2\}$, let $h_i$ be the last vertex common 
to the paths from $w_i$ to $z_i$ and from $v_Z$ to $z_i$. Since $w_i$ and $v_Z$ 
are incomparable in $N$, $h_i$ is a (not necessarily strict) descendant of $v_i$. 
In particular, $w_i$ and $h_i$ are distinct. We conclude the proof of the claim
by considering two possible cases: $w_1$ and $w_2$ are incomparable in $N$, or one is an ancestor of the other. 

If $w_1$ and $w_2$ are incomparable in $N$, then $w_1,h_1,v_Z,h_2,w_2,h_y$ is a $3$-alternating cycle of $N$, 
where $h_y$ is the last vertex common to the directed paths from $w_1$ to $y$ and from $w_2$ to $y$. 
In view of Proposition~\ref{pr-arbalt} his is impossible since $N$ is arboreal. If one of $w_1, w_2$ is an ancestor of the other, say $w_1$ 
is an ancestor of $w_2$ in $N$, then $w_1$ is an ancestor of $h_2$ in $N$, and $w_1,h_1,v_Z,h_2$ is a $2$-alternating cycle of $N$. 
Then the same argument as before shows that this  is impossible. This concludes the proof of the claim.

Now by the claim it follows that for all $x \in C(v_Z)$ and all $Y \in K_Z$, $x$ shares an ancestor with all elements of $Y$. 
Hence $Y \cup \{x\}$ is a clique in $G$ for all such $Y$. Since for all such $Y$, we have that  $Y\in K$, it follows that $x \in Y$. Thus
 $C(v_Z) \subseteq Y$ for all $Y \in K_Z$, and so $C(v_Z) \subseteq \bigcap_{Y \in K_Z} Y=Z$.
\end{proof}

We now prove two useful lemmas which provide more information 
concerning the sets $C(v)$ for $v$ a vertex in an arboreal network.	

\begin{lemma}\label{lm-merge}
Let $N$ be an arboreal network  and let $u,v \in V(N)$ distinct. Then the following hold:
\begin{itemize}
\item[(i)] If $u$ is an ancestor of $v$ in $N$, then $u$ has exactly one child that is an ancestor of $v$. 
Moreover, all other children $u'$ of $u$ satisfy $C(u') \cap C(v)=\emptyset$.
\item[(ii)] If $C(v) \subseteq C(u)$ and $u$ and $v$ are incomparable in $N$, 
then there exists a non-leaf descendant $h$ of both $u$ and $v$ satisfying $C(h)=C(v)$.
\end{itemize}
\end{lemma}

\begin{proof}
(i) To see the first part of the statement, suppose for contradiction that $u$ has two 
distinct children $u_1,u_2$ that are both ancestors of $v$. Then there exists a vertex $h$ in $N$
that is an ancestor of $v$, and is a descendant of both $u_1$ and $u_2$. 
Choosing $h$ in such a way that no strict ancestor of $h$ is a descendant of both $u_1$ and $u_2$, it follows
that $u,h$ is a 1-alternating cycle of $N$. In view of Proposition~\ref{pr-arbalt}, 
this is impossible since $N$ is arboreal. Hence, $u$ has exactly one child that is an ancestor of $v$.

To see the second part of the statement, let $u'$ be a child of $u$ that is not an ancestor of $v$, and 
let $x \in C(u')$. If $x \in C(v)$, then $x$ is a descendant of both $u'$ and $v$ in $N$.
Hence, there exists a vertex $h$ that is an ancestor of $x$ in $N$, and a descendant of both $u'$ and $v$. 
Choosing $h$ in such a way that no strict ancestor of $h$ is a descendant of both $u'$ and $v$, 
it follows that $u,h$ is a 1-alternating cycle of $N$.  Since $N$ is arboreal this is impossible in
view of Proposition~\ref{pr-arbalt}.  Hence, $C(u') \cap C(v)=\emptyset$.

(ii) Since $u$ and $v$ are incomparable, for all $z \in C(v)$, there exists a vertex $h_z$ that 
is an ancestor of $z$, and a descendant of $u$ (since $C(v)\subseteq C(u)$) and $v$.
Without loss of generality, we can choose $h_z$ in such a way that no strict ancestor of $h_z$ 
is a descendant of both $u$ and $v$.  Note that $h_z$ must be a hybrid vertex of $N$. In particular, 
it cannot be a leaf of $N$.  

We claim that  $C(h_z)=C(v)$, for any $z\in C(v)$. To see this, assume for contradiction that 
there exists $x,y \in C(v)$ distinct such that $h_x \neq h_y$. 
Then $u,h_x,v,h_y$ is a 2-alternating cycle of $N$ which is impossible in view of Proposition~\ref{pr-arbalt} as $N$ is arboreal. 
Hence, $h_x=h_y$, for all $x,y\in C(v)$. Choose some $x\in C(v)$. Then $C(v) \subseteq C(h_x)$ by 
the previous argument. Moreover, since $h_x$ is a descendant of $v$, we also 
have $C(h_x) \subseteq C(v)$ which completes the proof of the claim and also the proof of the lemma.
\end{proof}

\begin{lemma}\label{lm-inj}
Let $N$ be an arboreal network. If $N$ has no vertex of outdegree $1$ whose unique child 
is a non-leaf vertex then $C(u) \neq C(v)$,  for all internal vertices $u,v$ of $N$ distinct.
\end{lemma}

\begin{proof}
Assume for contradiction that there exist internal vertices  $u$  and $v$ in $N$ distinct such that $C(u)=C(v)$. 
Note that we may assume that $u$ and $v$ are such that $u$ is a strict ancestor of $v$ in $N$ (indeed, if $v$ is 
an ancestor of $u$ in $N$, then the roles of $u$ and $v$ can be reversed). 
If $u$ and $v$ are incomparable in $N$, then by Lemma~\ref{lm-merge}(ii), 
there exists a non-leaf vertex $h$ that is a descendant of both $u$ and $v$ in $N$ and 
satisfies $C(h)=C(v)=C(u)$. In this case, $h$ can play the role of $v$.

Since $u$ is a strict ancestor of $v$ in $N$ and $v$ is not a leaf, $u$ 
has outdegree at least $2$. Combined with Lemma~\ref{lm-merge}(i), it follows 
that there exists a child $u'$ of $u$ in $N$ that is not an ancestor of $v$ and 
for which $C(u')\cap C(v)=\emptyset$ holds. However, since $u'$ is a child of $u$, we 
also have $C(u') \subseteq C(u)=C(v)$ which is impossible. Hence, no two such elements $u$ and $v$ can exist.
\end{proof}

Now, recall from  Section~\ref{sc-sym} that a labelled arboreal network $(N,t)$ 
is \emph{discriminating} if $N$ has no internal arc $(u,v)$ such 
that $u$ has outdegree $1$, and no internal arc $(u,v)$ such that $v$ has 
indegree $1$ and $t(u)=t(v)$. This
definition is motivated by the fact that, for $(N,t)$ a labelled arboreal network, 
the labelled arboreal network $(N',t')$ obtained from $N$ by successively
applying the following operations to internal arcs $(u,v)$:
\begin{itemize}
\item[$\bullet$] If $u$ has outdegree 1 then collapse $(u,v)$ into a new vertex $w$. If $v$ had outdegree $2$ or more, put $t'(w)=t(v)$.
\item[$\bullet$] If $v$ has indegree $1$ and $t(u)=t(v)$ then collapse $(u,v)$
	into a new vertex $w$ and put $t'(w)=t(v)$.
\end{itemize}
and putting $t'(v)=t(v)$ for all other vertices $v$ satisfies $d_{(N',t')}=d_{(N,t)}$.
Note that, in a discriminating labelled arboreal network $(N,t)$, a vertex $v$ of $N$ has 
outdegree $2$ or more if and only if $|C(v)| \geq 2$. In particular, the labelling map $t$ assigns an element of $M$ to all such vertices.

We now prove a result which, for a labelled arboreal network $(N,t)$,
relates the sets $C(v)$ for $v$ a vertex in $N$ with properties of the map $d_{(N,t)}$.
First we require some further terminology.
Let $d: {X \choose 2} \to M^{\odot}$ be a symbolic map. We say 
that a non-empty subset $Y$ of $X$ is a \emph{clique-module} of $d$ if $|Y|=1$, or 
if $Y$ is a clique in $G_d$, and for all $x,y \in Y$ and all $z \in X-Y$ we 
have $|\{d(x,z),d(y,z), \odot\}| \leq 2$. Informally speaking, the latter means 
that if both $d(x,z)$ and $d(y,z)$ are elements in $M$ then $d(x,z)=d(y,z)$. We 
say that a clique-module $Y$ is \emph{trivial} if $|Y|=1$, and that 
it is \emph{strong} if for all clique-modules $Y'$ of $d$ such that $Y' \cup Y$ 
is a clique in $G_d$, $Y \cap Y' \in \{Y,Y',\emptyset\}$. Note that trivial 
clique-modules are always strong. We denote by $\mathcal M(d)$ the set of all strong, 
non-trivial clique-modules of $d$. To illustrate these notions, let $X=\{x,y,z,t,u\}$, 
and consider the map $d: {X \choose 2} \to \{\bullet, \circ, \odot\}$ 
defined by $d(x,z)=d(x,t)=d(y,z)=d(y,t)=d(z,t)=\bullet$, $d(x,y)=d(t,u)=\circ$, 
and $d(x,u)=d(y,u)=d(z,u)=\odot$. Then the non-trivial clique-modules of $d$ 
are $\{x,y,z,t\}$, $\{x,y\}$, $\{x,y,z\}$, $\{x,y,t\}$, $\{z,t\}$ and $\{t,u\}$. 
Of these, only $\{x,y,z,t\}$, $\{x,y\}$ and $\{t,u\}$ are strong.

\begin{proposition}\label{pr-cm2}
Let $(N,t)$ be a labelled arboreal network on $X$. For all vertices $v$ of $N$, $C(v)$ 
is a clique-module of $d=d_{(N,t)}$. Moreover, for all $Y \in \mathcal M(d)$, 
there exists a vertex $v$ of $N$ such that $C(v)=Y$.
\end{proposition}

\begin{proof}
We begin by proving the first statement in the proposition.
Let $v$ be a vertex of $N$. If $|C(v)|=1$, then $C(v)$ is a trivial clique-module of $d$. 
Hence, we may assume from now on that $|C(v)| \geq 2$.

By definition of $d$, $C(v)$ is a clique in $G_d$. Now, let $x,y \in C(v)$ distinct, 
and let $z \notin C(v)$ such that $\odot \notin \{d(x,z),d(y,z)\}$. Then, 
the vertex $\mathrm{lca}_N(x,y)$ is a descendant of $v$ in $N$, while the vertices 
$\mathrm{lca}_N(x,z)$ and $\mathrm{lca}_N(y,z)$ are not. Since these three least common ancestors  cannot be pairwise distinct, $\mathrm{lca}_N(x,z)=\mathrm{lca}_N(y,z)$, and so $d(x,z)=t(\mathrm{lca}_N(x,z))=t(\mathrm{lca}_N(y,z))=d(y,z)$. 
Hence, $C(v)$ is a clique-module of $d$.

To see that the second statement in the proposition holds, let $Y$ by a strong, non-trivial 
clique-module of $d$. By Lemma~\ref{lm-clique}, there exists a vertex $v$ of $N$ such 
that $Y \subseteq C(v)$. Without loss of generality, we may choose $v$ in such a way 
that no child of $v$ enjoys this property. We now show that $C(v) \subseteq Y$ also holds, 
so that $C(v) =Y$ which concludes the proof of the proposition.

By choice of $v$, there exist two distinct children $v_1,v_2$ of $v$ 
such that $C(v_1) \cap Y \neq \emptyset$ and $C(v_2) \cap Y \neq \emptyset$. 
Note that since $N$ is arboreal, Proposition~\ref{pr-arbalt} implies that $C(v_1) \cap C(v_2)=\emptyset$. 
Now, let $C'=C(v)-C(v_1)$. Since $C'$ is a subset of $C(v)$, $C'$ is a clique in $G_d$. 
We next claim that $C'$ is a clique-module of $d$. Let $x,y \in C'$, $z \notin C'$.  
In view of the first part of the proposition, $C(v)$ is a clique-module of $d$, so if $z \notin C(v)$, 
we have $d(x,z)=d(y,z)$. If $z \in C(v)$, then since $z \notin C'$, we have $z \in C(v_1)$. Hence,
 $\mathrm{lca}(x,z)=\mathrm{lca}(y,z)=v$ and so $d(x,z)=d(y,z)$. Thus, $C'$ is a clique-module of $d$, as claimed.

Since $Y$ is a strong non-trivial clique-module of $d$, we have $C' \cap Y \in \{C',Y,\emptyset\}$. Since
$C(v_1) \cap Y \neq \emptyset$,  we have that $Y \subseteq C'$ does not hold. 
Moreover, since $C(v_2) \cap Y \neq \emptyset$ and $C(v_2) \subseteq C'$ it follows that $Y \cap C' =\emptyset$ 
does not hold either. Hence, $C'=C(v)-C(v_1) \subseteq Y$. Replacing $v_1$ with $v_2$ 
in the latter argument, implies that $C(v)-C(v_2) \subseteq Y$ also holds. 
Thus, $C(v) \subseteq Y$, as required.
\end{proof}

\noindent Putting together the above results, we now prove a key  theorem 
that enables us to prove Theorem~\ref{cor-discr}.

\begin{theorem}\label{th-discr}
Let $(N,t)$ be a labelled arboreal network on $X$ and $d=d_{(N,t)}$. Then the following statements are equivalent:
\begin{itemize}
\item[(i)] $(N,t)$ is discriminating.
\item[(ii)] The map $\phi:V(N)-X \to \mathcal C(K(G_d)) \cup \mathcal M(d)$ given by $\phi(v) =C(v)$, for all $v\in V(N)-X$, 
is a bijection between $V(N)-X$ and $\mathcal C(K(G_d)) \cup \mathcal M(d)$.
\end{itemize}
\end{theorem}

\begin{proof}
To ease notation, set $K=K(G_d)$.

(i) $\Rightarrow$ (ii) We first show that, if $v \in V(N)-X$  then (at least) 
one of $C(v) \in \mathcal C(K)$ or $C(v) \in \mathcal M(d)$ must hold.  
By Proposition~\ref{pr-cm2}, $C(v)$ is a clique-module of $d$. If $v$ is a root of $N$, then 
$C(v) \in K \subseteq \mathcal C(K)$ (in fact $C(v) \in \mathcal M(d)$ also holds). 
If $v$ has indegree $2$ or more in $N$, then $C(v)=\bigcap_{C(v) \subset Y \in K}Y$. 
Hence, $C(v) \in \mathcal C(K)$ holds in this case too.

So, suppose $v$ has indegree $1$ in $N$. Then since $v\not\in L(N)$,
the outdegree of $v$ in $N$ must be at least 2. Hence, $v\in V(N)^-$. 
Furthermore, since the unique parent $u$ of $v$ in $N$ cannot be a leaf either,
$(u,v)$ must be an internal arc of $N$.	 Since $(N,t)$ is discriminating it 
follows that the  outdegree of $u$ is at least 2.  Hence,  $u\in V(N)^-$ also holds.

We next claim that $C(v)\in \mathcal M(d)$, that is, $C(v)$ is  a strong clique-module for $d$. 
Suppose for contradiction that $C(v)$ is not a strong clique-module for $d$, that is, 
there exists a clique-module $Y$ of $d$, such that $Y \cup C(v)$ is a clique in $G_d$
and $Y \cap C(v) \notin \{Y,C(v), \emptyset\}$.
Since $N$ is arboreal, $G_d$ and $\mathcal A(N)$ are isomorphic in view of Lemma~\ref{lm-du}. Since $|Y\cup C(v)|\geq 2$,
 Lemma~\ref{lm-clique} implies that there exists a vertex $w$ such that $Y \cup C(v) \subseteq C(w)$.
Without loss of generality, we may choose $w$ in such a way that no strict descendant of $w$ 
has this property. In view of Lemma~\ref{lm-merge}(ii), we may also assume 
that $w$ is an ancestor of $v$. Since $Y \nsubseteq C(v)$ as $C(v)$ is not a 
strong clique-module for $d$, it follows that $w$ is a strict ancestor of $v$. 
In particular, $w$ has outdegree $2$ or more. Thus, $w\in V(N)^-$.

We next show that $w\not=u$ and that $t(w)=t(v)$. To this end, note that by the choice of $w$ there exists $y \in Y$ such that $\mathrm{lca}_N(x,y)=w$ for 
all $x \in C(v)$. Now, let $x \in C(v)$ and $z \in C(v) \cap Y$ such that $x \notin Y$ and $\mathrm{lca}_N(x,z)=v$. 
Note that such an $x$ and $z$ always exists since,	by the choice of $Y$,  there 
always exist some $a \in C(v)-Y$ and $b\in C(v) \cap Y$. If $\mathrm{lca}_N(a,b)=v$ 
then we take $x=a$ and $z=b$. Otherwise, $\mathrm{lca}_N(a,b)$ must be a strict descendant of $v$. In that
case, we can choose some $c\in C(v)$ such that $c$ and $\mathrm{lca}_N(a,b)$
are descendants of different children of $v$. If $c\in Y$ then we can take $z$
to be $c$ and $x$ to be $a$, and if $c\not\in Y$ then we can take $x$
to be $c$ and $z$ to be $b$.
Since $Y$ is a clique-module of $d$  and neither $d(x,y)= \odot$ nor $d(x,z)= \odot$ holds as $x,y,z\in C(v)$,
we obtain $d(x,y)=d(x,z)$. Since $(N,t)$ explains $d$, it follows 
that  $t(w)= d(x,y)=d(x,z)=t(v)$, as required. Since $t(u)\not=t(v)$ because $(N,t)$ is discriminating, $w \neq u$ follows, also as required.

Now, let $p \in C(u)$ with $p \notin C(v)$. Then $\mathrm{lca}_N(x,p)=\mathrm{lca}_N(z,p)=t(u)$. 
If $p \in Y$ held, then $d(x,p)=d(x,z)$ since $Y$ is a clique-module of $d$ and neither 
$d(x,p)\not=\odot$ nor $d(y,p)\not=\odot$ holds. But this
is impossible, since $d(x,p)=t(u)\not=t(v)=d(x,z)$. Hence, $p \notin Y$. 
Similar arguments as in the case that $p\in Y$ imply that  $d(z,p)=d(y,p)$. 
But this is also impossible, since $t(u)\not=t(v)=t(w)=d(y,p)=d(z,p)=t(u)$. 
Thus, $C(v)\in\mathcal M(d)$, as claimed.

It remains to show that the map $\phi$
is bijective. That  $\phi$ is surjective is a direct consequence of 
Propositions~\ref{pr-foset} and~\ref{pr-cm2}. That $\phi$ is 
injective is a direct consequence of Lemma~\ref{lm-inj} since $(N,t)$ is 
discriminating and so $N$ does not contain an internal arc $(u,v)$ such 
that $u$ has outdegree $1$.

(ii) $\Rightarrow$ (i) We first remark that $N$ cannot have an internal arc $(u,v)$ 
such that $u$ has outdegree $1$. Indeed, if $N$ had such an arc, then $C(u)=C(v)$ 
would hold which contradicts the injectivity of $\phi$. To see that $N$ is 
discriminating, we therefore need to show that if $(u,v)$ is an internal arc of $N$ 
such that $v$ has indegree $1$ then $t(u) \neq t(v)$.

So, let $(u,v)$ be an internal arc of $N$ such that $v$ has indegree $1$. 
Since $\phi$ is injective and so  $C(w)\not=C(v)$ holds for all
vertices $w\in V(N)$, it follows that $C(v) \notin \mathcal C(K)$. 
Hence, $C(v) \in \mathcal M(d)$, that is, $C(v)$ is a strong clique-module of $d$. 
Now, let $v'$ be a child of $v$ which exists because $v$ is an internal vertex of $N$. 
Let $Y=C(u)-C(v')$. Note that since $v$ has indegree $1$, $v$ has outdegree $2$ or more. 
In particular, $v'$ is not the only child of $v$. Clearly, $Y$ is a clique in $G_d$. 
Since $C(u) \neq C(v)$, we have $Y \cap C(v)=C(v)-C(v') \notin \{Y,C(v),\emptyset\}$. 
Combined with the fact that $C(v)$ is a strong clique-module of $d$ it follows 
that $Y$ cannot be a clique-module of $d$. Hence, there must exist 
three elements $x_0,y_0 \in Y$, $z_0 \in X-Y$ such that $\odot \notin \{d(x_0,z_0),d(x_0,z_0)\}$ 
and $d(x_0,z_0) \neq d(y_0,z_0)$.

Since, by Proposition~\ref{pr-cm2}, $C(u)$ is a clique-module of $d$, we have 
for all $x,y \in Y \subseteq C(u)$ distinct and all $z \in X-C(u)$, that $|\{d(x,z),d(y,z),\odot\}| \leq 2$.  
Hence, $z_0 \in C(u)-Y=C(v')$. Since,
for all $x,y \in C(v)-C(v')$, we have $\mathrm{lca}_N(x,z)=\mathrm{lca}_N(y,z)=v$, it 
follows that $d(x,z)=d(y,z)=t(v) \neq \odot$. Similar arguments imply 
that, for all $x,y \in C(u)-C(v)$, $\mathrm{lca}_N(x,z)=\mathrm{lca}_N(y,z)=u$. 
Thus, $d(x,z)=d(y,z)=t(u) \neq \odot$ holds too. Hence, we 
must have (up to permutation) $x_0 \in C(u)-C(v)$ and $y_0 \in C(v)-C(v')$. 
In particular, we have $d(x_0,z_0)=t(u)$ and $d(y_0,z_0)=t(v)$. Since $d(x_0,z_0) \neq d(y_0,z_0)$, 
we have $t(u) \neq t(v)$, as required.
\end{proof}

\noindent{\em Proof of Theorem~\ref{cor-discr}.}
In view of Theorem~\ref{th-discr}, for two discriminating labelled arboreal networks $(N,t)$ and $(N',t')$ to
both explain $d$, there must exist a bijection $\psi:V(N) \to V(N')$ that is the identity on $X$
and such that $C(v)=C(\psi(v))$, for all $v \in V(N)$. It therefore suffices to show that 
(a) for all $u,v\in V(N)$ distinct, $(u,v)$ is an arc of $N$ if and only if $(\psi(u),\psi(v))$ 
is an arc of $N'$, and (b) for all internal vertices $v$ of $N$ of outdegree 2 or more, $t(v)=t'(\psi(v))$.

(a) Let $u,v\in V(N)$ distinct. By symmetry, it suffices to show that, if $(u,v)$ is 
an arc of $N$ then $(\psi(u),\psi(v))$ is an arc of $N'$. Clearly, $u$ is an internal 
vertex of $N$ and $C(v) \subseteq C(u)$. If $v$ is also an internal vertex of $N$, 
then Lemma~\ref{lm-inj} together with Lemma~\ref{lm-merge}(ii) imply that $\phi(u)$ 
is an ancestor of $\psi(v)$ in $N'$.
If $v$ is not an internal vertex of $N$, then it must be a leaf of $N$. Hence,
$\psi(v)=v\in C(u)=C(\psi(u))$. Consequently, $\psi(u)$ must also be an ancestor of $\psi(v)$ in this case.
To see that $\psi(u)$ is in fact a parent of $\phi(v)$, suppose for contradiction 
that there is a vertex $w \in V(N)$ distinct from $u$ and $v$ such that $\psi(w)$ 
lies on the directed path from $\phi(u)$ to $\psi(v)$ in $N'$. Combined 
with the definition of $\psi$, it follows that  $C(v) \subsetneq C(w) \subsetneq C(u)$. 
Since $u$ and $w$ cannot be leaves of $N$, Lemma~\ref{lm-inj} and 
Lemma~\ref{lm-merge}(ii) imply that $u$ is an ancestor of $w$ and, 
in case $v$ is not a leaf of $N$ either, that $w$ is an ancestor of $v$ in $N$. 
If $v$ is a leaf then similar arguments as before imply that $w$ is an ancestor of $v$.
Since $(u,v)$ is an arc of $N$, it follows that $u,v$ is a 1-alternating cycle of $N$. But this is impossible 
in view of Proposition~\ref{pr-arbalt} as $N$ is arboreal.
Thus such a vertex $w$ cannot exist and, so, $(\psi(u),\psi(v))$ must be an arc of $N'$.

(b) Assume that $v$ is an internal vertex of $N$ that has outdegree 2 or more. Since $(N,t)$ is discriminating,
$|C(v)| \geq 2$ must hold since otherwise $N$ would have a 1-alternating cycle which
is impossible in view of Proposition~\ref{pr-arbalt} because $N$ is arboreal. 
Hence, $t(v)=d_{(N,t)}(x,y)$ holds for all $x,y \in C(v)$ for which $\mathrm{lca}_N(x,y)=v$, 
and $t'(\psi(v))=d_{(N',t')}(x',y')$ holds for all $x,'y' \in C(v)$ for which $\mathrm{lca}_{N'}(x',y')=\psi(v)$. 
Since, by (a), the map $\psi$ is a graph isomorphism from $N$ to $N'$ that is the identity on $X$,
it follows that if $\mathrm{lca}_N(x,y)=v$, then $\mathrm{lca}_{N'}(x,y)=\psi(v)$. Hence, $t(v)=t'(\psi(v))$.
\qed\\

\end{document}